\documentclass[12pt]{amsart}
\usepackage[]{hyperref}
\usepackage{amssymb}
\usepackage{mathtools}
\usepackage{pythonhighlight}
\usepackage{enumitem}
\usepackage{tikz}
\usepackage{pgf}
\usepackage{xcolor}
\usetikzlibrary{calc,intersections,decorations.markings,cd,braids}

\newcommand{\C}{\mathbb C}
\newcommand{\Z}{\mathbb Z}
\newcommand{\Q}{\mathbb Q}

\DeclareMathOperator{\pic}{Pic}
\DeclareMathOperator{\qhd}{\Q HD}
\newtheorem{theorem}{Theorem}[section]

\newtheorem{proposition}[theorem]{Proposition}
\newtheorem{lemma}[theorem]{Lemma}
\newtheorem{corollary}[theorem]{Corollary}
\newtheorem{theorem0}{Theorem}
\theoremstyle{definition}

\newtheorem{example0}[theorem0]{Example}

\newtheorem*{ack}{Acknowledgments}
\theoremstyle{remark}
\newtheorem{remark}[theorem]{Remark}
\numberwithin{equation}{section}

\allowdisplaybreaks
\begin{document}
\title[
Fundamental group of rational homology disk smoothings]
{
Fundamental group of rational homology disk smoothings of surface singularities}

\author[E. Artal]{Enrique Artal Bartolo}
\address{Departamento de Matem\'aticas-IUMA, Facultad de Ciencias,
Universidad de Zara\-goza,
c/ Pedro Cerbuna 12\\
E-50009 Zaragoza SPAIN}
\email{artal@unizar.es}

\author[J. Wahl]{Jonathan Wahl}
\address{Department of Mathematics\\The University of
North Carolina\\Chapel Hill, NC 27599-3250}
\email{jmwahl@email.unc.edu}

\thanks{First named author is partially supported by 
MCIN/AEI/10.13039/501100011033
(MTM2016-76868-C2-2-P, PID2020-114750GB-C31) and by Departamento de Ciencia, Universidad y Sociedad del Conocimiento of the
Gobierno de Arag{\'o}n (E22\_20R: ``{\'A}lgebra y Geometr{\'i}a'').}

\keywords{surface singularity,
rational homology sphere, Milnor fibre}
\subjclass{14H20, 32S50, 57M05}

\begin{abstract}  It is known (\cite{SSW}) that there are exactly three triply-infinite and seven singly-infinite families of weighted homogeneous normal surface singularities admitting a rational homology disk ($\qhd$) smoothing, i.e., having a Milnor fibre with Milnor number zero.  Some examples are found by an explicit ``quotient construction'', while others require the ``Pinkham method''.  The  fundamental group of the Milnor fibre has been known for all except the three exceptional families $\mathcal B_2^3(p), \mathcal C^3_2(p),$ and $\mathcal C^3_3(p)$.  In this paper, we settle these cases.  We present a new explicit construction for the $\mathcal B_2^3(p)$ family, showing the fundamental group is non-abelian (as occurred previously only for the $\mathcal A^4(p), \mathcal B^4(p)$ and $\mathcal C^4(p)$ cases).  We show that the fundamental groups for $ \mathcal C^3_2(p)$ and $\mathcal C^3_3(p)$ are abelian, hence easily computed; using the Pinkham method here requires precise calculations for the fundamental group of the complement of a plane curve.

  \end{abstract}
\maketitle

\section*{Introduction}

Let $(X,0)$ be the germ of a complex normal surface singularity, with neighborhood boundary (or link) $\Sigma$.
A  \emph{smoothing of $(X,0)$} is a morphism $f:(\mathcal{X},0)\rightarrow (\mathbb C,0)$, with $(\mathcal{X},0)$ a three-dimensional isolated Cohen-Macaulay singularity, equipped with an isomorphism $(f^{-1}(0),0)\simeq (X,0)$.  The Milnor fibre $M$ is the general fibre $f^{-1}(\delta)$, a $4$-manifold with boundary $\Sigma$. The second Betti number of $M$ is called $\mu$, the Milnor number of the smoothing; the first Betti number always vanishes \cite{gst}.   We say $f$ is a  $\qhd$ (or \emph{rational homology disk}) smoothing if $\mu=0$, i.e., the Euler characteristic $\chi(M)=1$. In such a case, the $3$-manifold $\Sigma$ has a particularly interesting filling (e.g., it is Stein).

\begin{example0}
Such smoothings occur for cyclic quotient singularities of type $\frac{n^2}{nq-1}\equiv
\frac{1}{n^2}(1,n q-1)$, where $0<q<n, (n,q)=1$ (\cite[(2.7)]{w2}).   
One proceeds as follows; if  $f(x,y,z)=xz-y^n$, then $f:\C^3\rightarrow \mathbb C$ is a smoothing of the $A_{n-1}$ singularity, whose Milnor fibre $M$ is simply connected, with Euler characteristic $n$.   Let  $G\subset GL(3,\C)$ by the diagonal cyclic group generated by   $[\zeta, \zeta^q, \zeta^{-1}]$, where $\zeta=\exp\frac{2\pi i}{n}$. 
The group $G$ acts freely on $\C^3\setminus\{0\}$ and $f$ is $G$-invariant; so the induced map $f:\C^3/G\rightarrow \C$ is a smoothing of the cyclic quotient singularity $A_{n-1}/G$, which has type $\frac{n^2}{nq-1}$. The new Milnor fibre is the free quotient $M/G$, of Euler characteristic $1$, hence is a $\qhd$.  We call this class~$\mathcal G_{n,q}$.
\end{example0}

In the early 1980's, the second named author produced $9$ other families (e.g., \cite[(5.9.2)]{w3}, but mainly unpublished).   As with the family $\mathcal G_{n,q}$, some examples can be produced by an explicit ``quotient construction''.  

Start with a smoothing $f:(\mathcal Y,0)\rightarrow (\C,0)$ of some $2$-dimensional germ $(Y,0)$, with simply connected Milnor fibre $M$.  Assume $G$ is a group of automorphisms of $(\mathcal Y,0)$ acting fixed point freely off $0$, with $f$ $G$-invariant.  Then $f:(\mathcal Y/G,0)\rightarrow (\C,0)$ is a smoothing of $Y/G$ with Milnor fibre $M/G$.  If $\chi(M)=|G|$, then $\chi(M/G)=1$, so $M/G$ is a $\qhd$.  

This method can produce equations for the families eventually named $\mathcal W(p,q,r)$, $\mathcal N(p,q,r)$, $\mathcal A^4(p)$, $\mathcal B^4(p)$, and $\mathcal C^4(p)$ (here $p,q,r \geq 0$); the $\mathcal Y$ involved could be $\C^3$, a hypersurface singularity in $\C^4$, or the cone over a del Pezzo surface in $\mathbb P^6$.  All these singularities are weighted homogeneous, and a superscript denotes the valence of the central curve in the graph of the minimal good resolution.

It turns out that other (and in fact all) weighted homogeneous $\qhd$ examples can be constructed by using the Pinkham method of ``smoothing of negative weight'' \cite{p3},~\cite{w6}:
\begin{itemize}
\item Consider the projective $\C^*$-compactification of the singularity;
\item resolve at infinity to obtain a curve configuration $E'$;
\item if possible, smooth the projective surface keeping $E'$ fixed.  
\item The projective general fibre $Z$ is a rational surface containing a configuration $D'$ isomorphic to $E'$, obtained by blowing-up $\mathbb P^2$ along an appropriate plane curve $D$ so that $D'$ is the total transform of $D$ minus several curves. 
\item Then $Z\setminus D'$ is the Milnor fibre of the smoothing.    
\end{itemize}

The Milnor fibre is a $\qhd$ when the components of $D'$ rationally span $\pic Z$.  Given some cohomological vanishing conditions, Pinkham's construction allows one to go backwards from a given pair $(Z,D')$ to a $\qhd$ smoothing of a weighted homogeneous surface singularity.  
The group $\pi_1(Z\setminus D')$ is difficult to compute in general, unless one already knows that $\pi_1(\mathbb P^2\setminus D)$ is abelian; this occurs for types $\mathcal W, \mathcal N$, and $\mathcal M$, since here $D$ can be taken to be four lines in general position.

The possible resolution graphs of any $(X,0)$ admitting a $\qhd$ smoothing were greatly restricted by the results of \cite{SSW}, which also gave names to the known examples.  For $(X,0)$ weighted homogeneous, these turned out to be the only ones, via:  

\begin{theorem0}[Bhupal-Stipsicz Theorem \cite{bs}]  The weighted homogeneous surface singularities admitting a $\qhd$ smoothing are the following families: 
$\mathcal G_{n,q}$, $\mathcal W(p,q,r)$, $\mathcal N(p,q,r)$, $\mathcal M(p,q,r)$, $\mathcal B_2^3(p)$, $ \mathcal C^3_2(p)$, $ \mathcal C^3_3(p)$, $\mathcal A^4(p)$, $ \mathcal B^4(p)$, and $\mathcal C^4(p)$.
\end{theorem0}

Some further results are as follows:
\begin{enumerate}
\item By earlier results of Laufer (\cite{laufert}), for the first seven families, the analytic type is uniquely determined by the graph of the singularity.  For the valence $4$ examples, there is in each case a unique cross-ratio for which a $\qhd$ smoothing exists (\cite{jf}).
\item In the base space of the semi-universal deformation of these singularities, a $\qhd$ smoothing component has dimension one, 
and there are one or two such components (\cite{jf}, \cite[(7.2)]{w7}).
\item The first homology of the Milnor fibre is isomorphic to a self-isotropic subgroup of $H_1(\Sigma)$, the discriminant group of the singularity (\cite{lw}).
\item Every $\qhd$ smoothing arises from a quotient construction $\mathcal Y\rightarrow \mathcal Y/G$, where $\mathcal Y$ is canonical Gorenstein and $G=\pi_1(M)$ \cite{w6}.  In general, the embedding dimension of $\mathcal Y$ can be arbitrarily large, so one cannot expect explicit equations, and only the Pinkham method seems available.
\item The fundamental group of the Milnor fibre is abelian for the families $\mathcal G$, $\mathcal W$, $\mathcal N$, $\mathcal M$, 
and non-abelian metacyclic for $\mathcal A^4, \mathcal B^4, \mathcal C^4$ (\cite{jf},~\cite{w5},~\cite{w7}).
\end{enumerate}
  
The main results of this paper consider the three exceptional cases for which $\pi_1(M)$ was unknown, namely $\mathcal B_2^3(p)$, $ \mathcal C^3_2(p)$, and $ \mathcal C^3_3(p)$.  
We state the results and list the resolution dual graphs
(with the convention that no weight means $-2$).

\begin{theorem0}\label{thm:c23}  The fundamental group of the $\qhd$ smoothing of $\mathcal C_2^3(p)$ is cyclic, of order $3(p+3)$.
\end{theorem0}

\begin{figure}[ht]
\begin{tikzpicture}[scale=1.5]
\foreach \a in {0,...,6}
{
\coordinate (A\a) at (\a,0);
}
\draw (A0)--(A2) (A4)--(A6);
\draw[dashed] (A2)--(A4);
\foreach \a in {0,1,2,4,5,6}
{
\fill (A\a) circle [radius=.1];
}
\draw (A5)--(5,1);
\fill (5,1) circle [radius=.1];
\node[above=4pt] at (A3) {$\overbrace{\hphantom{\hspace{3cm}}}^p$};
\node[above=2pt] at (A0) {$-(p+3)$};
\node[above=2pt] at (A6) {$-3$};
\node[left=2pt] at (5,1) {$-6$};
\end{tikzpicture}
\caption{Resolution dual graph of $\mathcal C_2^3(p)$}
\label{fig:C23-local}
\end{figure}
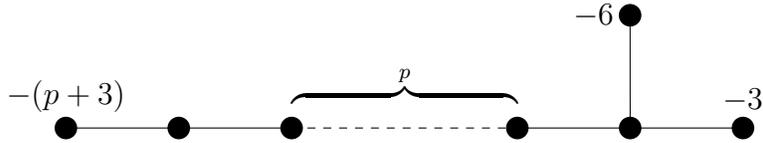

\begin{theorem0}\label{thm:c33}  The fundamental group of the $\qhd$ smoothing of $\mathcal C_3^3(p)$ is cyclic, of order $2(p+4)$.
\end{theorem0}

\begin{figure}[ht]
\begin{tikzpicture}[scale=1.5]
\foreach \a in {0,...,6}
{
\coordinate (A\a) at (\a,0);
}
\draw (A0)--(A2) (A4)--(A6);
\draw[dashed] (A2)--(A4);
\foreach \a in {0,1,2,4,5,6}
{
\fill (A\a) circle [radius=.1];
}
\draw (A5)--(5,1);
\fill (5,1) circle [radius=.1];
\node[above=4pt] at (A3) {$\overbrace{\hphantom{\hspace{3cm}}}^{p}$};
\node[above=2pt] at (A0) {$-(p+4)$};
\node[left=2pt] at (5,1) {$-6$};
\end{tikzpicture}
\caption{Resolution dual graph of $\mathcal C_3^3(p)$}
\label{fig:C33-local}
\end{figure}
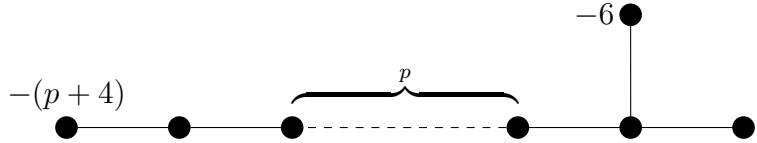

\begin{theorem0}\label{thm:b23}  
The fundamental group of the $\qhd$ smoothing of $\mathcal B_2^3(p)$ is non-abelian of order $4(p+2)(p+3)$, with an index $2$ cyclic subgroup and abelianization of order $4(p+3)$.  There is an explicit quotient construction, with $\mathcal Y$ a hypersurface singularity in $\C^4$.

\end{theorem0}

\begin{figure}[ht]
\begin{tikzpicture}[scale=1.5]
\foreach \a in {0,...,6}
{
\coordinate (A\a) at (\a,0);
}
\draw (A0)--(A2) (A4)--(A6);
\draw[dashed] (A2)--(A4);
\foreach \a in {0,1,2,4,5,6}
{
\fill (A\a) circle [radius=.1];
}
\draw (A5)--(5,1);
\fill (5,1) circle [radius=.1];
\node[above=4pt] at (A3) {$\overbrace{\hphantom{\hspace{3cm}}}^p$};
\node[above=2pt] at (A0) {$-(p+3)$};
\node[above=2pt] at (A1) {$-3$};
\node[above=2pt] at (A6) {$-4$};
\node[left=2pt] at (5,1) {$-4$};
\end{tikzpicture}
\caption{Resolution dual graph of $\mathcal B_2^3(p)$}
\label{fig:b23-local}
\end{figure}
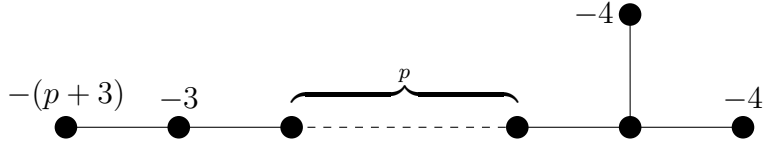

In all three cases, one can use the plane curves $D$ and their blow-ups to produce the pair $(Z,D')$ for which one must compute the fundamental group of the complement.
The proofs can be found at \S\ref{sec:b23}, \S\ref{sec:c23} and~\S\ref{sec:c33}. 

 For  $\mathcal B^3_2(p)$, the precise description of the fundamental group has allowed us
to find a direct quotient construction in \S\ref{sec:equations}.The first author initially did the computation in case $p=0$, discovering that the group was non-abelian, of order $24$.  The second author used this unexpected result to first construct a non-abelian cover of degree~$24$ of the original singularity,  a complete intersection in $\C^4$, and then to find a fixed-point free $4$-dimensional representation of the group leaving this cover and its smoothing invariant.  The first author later extended his computation of the fundamental group for all $p$, while independently the second author extended the explicit construction for all $p$, as in Theorem~\ref{thm:cyclic_construction} below.  Once a quotient construction for all $p$ is obtained, the results of Fowler \cite{jf}  imply there is only one $\Q$HD smoothing, so the fundamental group computations using $D$ become unnecessary.  Nonetheless, a detailed presentation of these results is included in \S\ref{sec:b23}, as the computational methods are important and illustrate a basic method.

\begin{ack}
This paper originated from fruitful conversations after a talk by the second author at the Némethi Conference in Budapest in May 2019. The scientific environment of this event, and the financial support given by the conference organizers, made possible initial discussions among the authors and eventually led to this work
\end{ack}

\section{The family \texorpdfstring{$\mathcal B^3_2(p)$}{B32p}}

From Figure~\ref{fig:b23-local} we see that the continued fraction expansion of the long arm, starting from the outside, arises from $2(p+2)^2/(2p+3)$, and the discriminant group has order $16(p+3)^2$; it follows that the first homology group of the Milnor fibre has order $4(p+3)$.

\begin{theorem}\label{thm:cyclic_construction}  For each $p\geq 0$, there is a hypersurface singularity $(\mathcal Y,0)\subset (\C^4,0)$, a group $G\subset SL(4,\C)$ acting freely on $\mathcal Y\setminus\{0\}$, and a $G$-invariant function $f$, so that $f:\mathcal Y/G\rightarrow \C$ provides the $\qhd$ smoothing of a singularity of type $\mathcal  B_2^3(p)$.
\end{theorem}

Writing down explicitly the equations and the representation of $G$, it is straightforward to construct a $\qhd$ smoothing whose Milnor fibre has $G$ as fundamental group.  What takes extensive computation is the verification that the smoothed singularity is of type $\mathcal  B_2^3(p)$.

\subsection{The group \texorpdfstring{$G$}{G}}\label{sec:group}
\mbox{}

Let $m\geq 2$ be an integer, $N=2m(m+1)$, $\omega$ a primitive $N^{th}$ root of $1$.  Consider the diagonal linear transformation of $\C^4$ given by $$S(a,b,c,d)=(\omega a,\  \omega^{-(2m+1)}b,\,\ \omega^{2m+1}c,\,  \ \omega^{-1}d).$$
The action $S$ can be also be written
\[
S=\frac{1}{N}[1,\ -(2m+1),\ 2m+1,\ -1].
\]
This allows one to quickly write down $S^p$ when $N=pN'$; replace $\frac{1}{N}$ by $\frac{1}{N'}$, and then reduce the entries in $[\bullet,\bullet,\bullet,\bullet ]
\bmod N'$.

Let $\zeta=\omega^{m}$ a primitive $(2m+2)^{th}$ root of $1$, and  define 
    $$T(a,b,c,d)=(\zeta b,a,d,\zeta^{-1}c).$$ 
One easily finds $$S^N=I,\  TST^{-1}=S^{-(2m+1)},\ T^2=S^{m}=\frac{1}{2m+2}[1,1,-1,-1].$$
Let $G=G_m\subset SL(4,\C)$ be the group generated by $S$ and $T$.  
\begin{proposition} The following properties hold for $G$:
\begin{enumerate}[label=\rm(G\arabic{enumi})]
\item $|G|=2N=4m(m+1)$.
\item The abelianization of $G$ has order $4(m+1)$, is cyclic when $m$ is odd, and is $\Z/(2(m+1)) \times \Z/(2)$ if $m$ is even.
\item The center of $G$ is the cyclic group generated by $S^{m}$, of order $2(m+1)$.
\item $S^iT$ has even order $> 2$.
\item $G$ acts freely on $\C^4$ off the origin.
\end{enumerate}

\begin{proof}  The first two statements are straightforward.  One sees that no $S^iT$ commutes with $S$, so the center is generated by a power of $S$, easily seen to be $S^{m}$.  The powers of $\omega$ that occur in $S$ are all primitive roots of $1$, so the subgroup generated by $S$ acts freely on $\C^4\setminus\{0\}$.  Note $$(S^iT)^2=S^{-m(2i-1)},$$ which can not equal the identity, so itself has no fixed points; thus, $S^iT$ has  no fixed points.
\end{proof} 

\end{proposition}
   
\begin{remark} One could also consider the simpler linear transformation 
$T'$ defined by $T'(a,b,c,d)=(b,-a,d,-c)$, and the group $G'\subset SL(4,\C)$ that $S$ and $T$ generate.  One now has 
$$S^N=I,  \ T'^2=S^{m(m+1)}=-I,\   T'ST'^{-1}=S^{-(2m+1)}.$$   When $m$ is even,  $G'$ is isomorphic to $G$, as seen by setting $T'=S^{\frac{m+2}{2}}T$; one may use this representation to consider later simpler polynomials $xw-yz$ and $zw-x^{2m}+y^{2m}$.
   
However, when $m$ is odd, $G'$ is not isomorphic to $G$, and does not act freely on $\C^4\setminus\{0\}$ because $S^{\frac{m+1}{2}}T'$ has order $2$ and fixed points.  The relevant representation of $G$ is now more complicated than in the even case.
   
   \end{remark}

   We make some additional remarks about the group $G$, which however are not used later on.

Define the \emph{generalized quaternion 2-group $Q_r$} by generators and relations as 
$$Q_r:  A^{2^{r-1}}=1, \ A^{2^{r-2}}=B^2, \ BAB^{-1}=A^{-1}.$$
Then $|Q_r|=2^r$, \  $Q_{r-1}\subset Q_r$ (use generators $A^2$ and $B$), and $Q_3$ is the usual quaternion group of order $8$.

\begin{proposition}  For $G$ as above, write $N=2m(m+1)=2^{r+1}p,$ where $p$ is odd.
\begin{enumerate}[label=\rm(Q\arabic{enumi})]
   \item\label{Q1} $H=\langle S^{2^{r+1}}\rangle$  is a cyclic normal subgroup of order $p$, consisting of all elements of odd order
   \item\label{Q2} If $m$ is odd,  write $m+1=2^r(2u-1)$, and $J=\langle S^uT\rangle$.  Then $J$ is a cyclic Sylow $2$-subgroup of order $2^{r+2}$, and $G$ is the semi-direct product of $H$ and $J$.
   \item\label{Q3} If $m=2^rq$ is even (with $q$ odd), the Sylow $2$-subgroup $$J=\langle S^{q(m+1)},S^{\frac{m+2}{2}}T\rangle$$ is isomorphic to $Q_{r+2}$, and $G$ is the semi-direct product of $H$ and $J$.
   \item\label{Q4} The $2$-Sylow subgroup of $G$  is normal if and only if $m$ is a power of $2$, in which case $G$ is the direct product of  $H=\langle S^m\rangle$ and $J=\langle S^{m+1}, S^{\frac{m+2}{2}}T\rangle$.
\end{enumerate}
   \end{proposition}

\begin{proof}  \ref{Q1} is straightforward.  For \ref{Q2}, note $S$ has order $N=2^{r+1}m(2u-1)$ and $(S^uT)^2=S^{-m(2u-1)}$, so $J$ is cyclic of order $2^{r+2}$, hence is a $2$-Sylow subgroup.  Since $G=HJ$, $H\cap J=\{I\}$, and $H$ is normal, one has a semi-direct product. 
     
In  \ref{Q3}, one checks that the given generators of $J$ match the generators and relations of $Q_{r+2}$; it follows as before that there is a semi-direct product decomposition.

For \ref{Q4}, normality of $J$ in the case of  $m$ odd would imply $G$ is abelian, which is never true.  For $m$ even, normality of $J$ implies that the conjugate $S\cdot S^{\frac{m+2}{2}}T\cdot S^{-1}$ is of the form $S^{iq(m+1)}\cdot S^{\frac{m+2}{2}}T$, for some $i$.  A calculation shows this is equivalent to $S^{(m+1)(qi-2)}=I$, so $2^{r+1}q$ divides $qi-2$.  Since $q$, which is odd, divides $2$, we have $q=1$, and one can set $i=2$.   One easily checks that the $T$-conjugate of $S^{\frac{m+2}{2}}T$ is also in $J$.
   \end{proof}

\begin{remark} The groups $G$, with a fixed-point free representation, have the familiar property (seen for instance in \cite{wolf}) that odd order 
Sylow subgroups are cyclic, and the $2$-Sylow is either cyclic or contains an index two cyclic subgroup.  An avid reader might  wish to locate the groups above in the complete chart in \cite[Section~7.2]{wolf}.

\end{remark}

\subsection{The equations}\label{sec:equations}
\mbox{}

The representation above of $G$ acting on $\C^4$ has its contragredient representation acting on the coordinate functions $x,y,z,w$, via
$$S(x,y,z,w)=(\omega^{-1}x,\ \omega^{2m+1}y,\ \omega^{-(2m+1)}z,\ \omega w)$$
$$T(x,y,z,w)=(\zeta^{-1}y,\ x,\ w,\ \zeta z).$$
\begin{proposition} The group $G$ acts freely on $\C^4$ off the origin, leaves invariant the hypersurface singularity 
$$\mathcal Y=\{zw+x^{2m}+\zeta y^{2m}=0\}\subset \C^4,$$ and fixes the polynomial
$$f(x,y,z,w)=xw+yz.$$ 
Thus $f:\mathcal Y/G\rightarrow \C$ is a smoothing of the $G$-quotient of the isolated complete intersection singularity $Y=\mathcal Y \cap \{f=0\}\subset \C^4$.
\begin{proof} The only new item needed is the simple calculation that $Y$ has an isolated singularity at the origin.
\end{proof}
\end{proposition}

The map $f:\mathcal Y \rightarrow \C$ gives a smoothing of $Y$.  By Hamm-L\^{e} (e.g., \cite{hamm}), the Milnor fibre $M=f^{-1}(\delta)$ is simply connected.  The Euler characteristic can be computed from the Greuel-Hamm formula \cite{gh} for weighted homogeneous complete intersections, yielding $$\chi(M)=1+\mu=4m(m+1).$$
The group $G$ acts freely on $\mathcal Y\setminus\{0\}$ and  $M$.  As $\chi(M)=|G|$, $M/G$ has Euler characteristic $1$, hence is a rational homology disk whose fundamental group is isomorphic to $G$.
 
 \begin{proposition} The map $f:\mathcal Y/G\rightarrow \C$ gives a rational homology disk ($\qhd$) smoothing of the singularity $Y/G$, whose Milnor fibre has non-abelian fundamental group $G$.
 \end{proposition}
The following section is devoted to the proof of the following proposition.

\begin{proposition} The singularity $Y/G$ is of type $\mathcal B_2^3(m-2).$
\end{proposition}

A priori, one knows the quotient is a rational singularity with discriminant the square
$[4(m+1)]^2$.  The standard approach (e.g., [12] or [17]) is to lift the action of $G$ from $Y$ to its \emph{Seifert partial resolution} $\mathcal S\rightarrow Y$, the result of weighted blow-up, which has a smooth central curve $C$ along which are cyclic quotient singularities.  The quotient $\mathcal S/G$ will be the Seifert resolution of $Y/G$.  
The resolution space $\mathcal S$ has a covering $\{\mathcal S_i\}$ by $4$ open affines, corresponding to weighted inversion of the coordinates.   One identifies singular points and fixed points of the action of $G$ along $C$ on each affine, as well as on certain partial quotients.  At the end, one finds three singular points on a rational curve, whose self-intersection on its minimal resolution is computable from knowledge of the discriminant.

\subsection{Resolution of \texorpdfstring{$Y/G$}{Y/g}}
\mbox{}

 If $\C^t$ has coordinates $z_i$ with positive integer weights $n_i$ (without common divisor), the weighted blow-up is a map $\mathcal U\rightarrow \C^t$, with fibre over the origin the weighted projective space $\mathbb P_{\textbf{n}}=\mathbb{P}_{(n_1,\cdots, n_t)}$. The space $\mathcal U$ has an open affine covering $U_i$, each of which is a quotient of an affine space $V_i$ by a cyclic group of order $n_i$.  For instance, $V_1$ has coordinates $A_1, \dots ,A_t$, related to the $z_i$ via 
$$
z_1=A_1^{n_1},\ z_2=A_1^{n_2}A_2,\dots,\ z_t=A_1^{n_t}A_t;
$$
the quotient $U_1$ equals $V_1$ modulo the action on the $A_i$'s of the cyclic group generated~by  
$$
\frac{1}{n_1}[-1,\ n_2,\dots,\ n_t].
$$  
Weighted blow-up of $\mathbb C^4$, with coordinates $x,y,z,w$ and weights $1,1,m,m$,  induces a weighted blow-up $\mathcal S \rightarrow Y$, covered by $4$ affines $\mathcal S_i$.  The exceptional fibre is a smooth projective curve $C\subset \mathbb P_{\textbf{m}}=\mathbb P_{(1,1,m,m)}$.

Inverting first $x$,  $\mathcal S_1$ has coordinates  $$x=A_1, \ y=A_1A_2,\ z=A_1^{m}A_3,\ w=A_1^{m}A_4,$$  with equations
\begin{align*}
A_4=&-A_2A_3\\
A_3A_4&=-(1+\zeta A_2^{2m}).
\end{align*}
Thus in coordinates $A_1, A_2, A_3$,  $\mathcal S_1$  is defined by 
$$
A_2A_3^2-1-\zeta A_2^{2m}=0,
$$
while $C$ is given by  $A_1=0$.  Both the surface and the curve are smooth.  
Rewriting  $C$ as 
$$
(A_2A_3)^2=A_2(1+\zeta A_2^{2m}),
$$ 
its function field is a double cover of the affine line branched at $2m+1$ points; so, $C$ is 
a hyperelliptic curve of genus $m$.

The group $G$ lifts to a group of automorphisms of $\mathcal S$.  On $\mathcal S_1$, 
$$
S(A_3)=S(z x^{-m})=(\omega^{-(2m+1)}z) (\omega^{-1}x)^{-m}=\omega^{-(m+1)}z x^{-2m}=\omega^{-(m+1)}A_3,
$$ 
while 
$$
T(A_3)=T(z) T(x)^{-m}=w (\zeta^{-1}y)^{-m}=\zeta^mA_4 A_2^{-m}=\zeta^{-1}A_3 A_2^{1-m}.
$$
We summarize as 
\begin{align*}
S(A_1,A_2,A_3)=&(\omega^{-1}A_1,\ \omega^{2m+2}A_2,\ \omega^{-(m+1)}A_3)\\
T(A_1,A_2,A_3)=&(\zeta^{-1}A_1A_2,\ \zeta A_2^{-1}, \ \zeta^{-1} A_3 A_2^{1-m}).
\end{align*}
The group $G$ acts on $\mathcal S_1$ ($A_2$ is never $0$ there).  Note $S^{2m}$ is a pseudo-reflection, sending $A_1$ to $\omega^{-2m}A_1$, leaving $A_2$ and $A_3$ fixed. Then $\bar{\mathcal S}_1=\mathcal S_1/\langle S^{2m}\rangle$ has as coordinates the invariants  $A_1^{m+1}\equiv \bar{A_1}$, $A_2$, and $A_3$, with equation $$A_2A_3^2=1+\zeta A_2^{2m}$$ and central curve $C$ defined by $\bar{A_1}=0.$  Then $\bar{G}=G/\langle S^{2m}\rangle$ 
acts on $\bar{\mathcal S_1}$ as follows: Let $\eta=\omega^{-(m+1)}$,  a primitive $(2m)^{\text{th}}$ root of $1$.  Then
\begin{align*}
S(\bar{A_1},A_2,A_3)=&(\eta \bar{A_1},\ \eta^{-2} A_2,\ \eta A_3)\\
T(\bar{A_1},A_2,A_3)=&(-\bar{A_1}A_2^{m+1},\ \zeta A_2^{-1},\ \zeta^{-1}A_3 A_2^{1-m}).
\end{align*}
We describe all fixed points of elements of $\bar{G}$ and their orbits.

First,  $S^{m}$ fixes all points of the form $(0,a,0)$, where $a$ is a $(2m)^{th}$ root of $-\zeta^{-1}=\zeta^m$.  Defining a square root of $\zeta$ by $\tau^2=\zeta$,  $a$ is of the form $\tau \eta^k$, $k=0,\cdots,2m-1$.  These $2m$ fixed points are permuted by powers of $S$, which sends for instance $a=\tau \eta^k$ to $\tau\eta^{k+2}$.  In particular, there are $2$ $G$-orbits of these fixed points, corresponding to $a=\tau$ and $a=\tau\eta$.  
      
A calculation shows that
$$
S^kT(\bar{A_1}, A_2,A_3)=(\eta^{m-k}\bar{A_1}A_2^{m+1},\ \zeta\eta^{2k} A_2^{-1},\ \zeta^{-1}\eta^{-k}A_3 A_2^{1-m})
$$
has two fixed points as above, where $a=\pm \tau \eta^k$; note that in $\bar{G}$, $(S^kT)^2=S^{m}$.  Thus the isotropy subgroup of $\bar{G}$ at $a=\tau$ is the cyclic group of order $4$ generated by $T$, and similarly the isotropy group at $a=\tau\eta$ is generated by $ST$.  At the point $(0,\tau,0)$, generators for the local ring are $\bar{A_1}$ and $A_3$, with $C$ given by $\bar{A_1}=0$.  As
$$
T(\bar{A_1},A_2,A_3)=(-\bar{A_1}A_2^{m+1},\ \zeta A_2^{-1},\ \zeta^{-1}A_3 A_2^{1-m}),
$$
the action on the tangent space when $A_2=\tau$ is checked to be scalar multiplication by $\tau^{-(m+1)}$, which is a primitive $4^{th}$ root of $1$.  The quotient is the singularity~$\frac{1}{4}[1,1]$ whose minimal resolution is a single smooth rational curve of self-intersection $-4$.   The same calculation holds at the fixed point when $A_2=\tau\eta$.  These two orbits are the only fixed points of $\bar{G}$ on $\bar{\mathcal S}_1$, so the quotient has two $(-4)$-singularities along the central curve.  The rationality of the central curve is known because one has a $\Q$HD smoothing, or can be seen directly as follows:
Invariants of $S^m$ on $\bar{\mathcal S_1}$ are $M=\bar{A_1}^2, N=\bar{A_1}A_3, P=A_3^2$, and $A_2$, so the quotient is defined by equations 
$$
MP=N^2, \ A_2P=1+\zeta A_2^{2m}.
$$ 
The image of  $C$ is given by $M=N=0$ and the plane curve $\ A_2P=1+\zeta A_2^{2m}$, which is clearly rational ($P$ is a function of $A_2$).  

 So, the quotient $\mathcal S_1/G$ consists of a surface with exactly two $(-4)$-singularities along a rational curve.

Since $\mathcal S_1 \cap C$ consists of all points of $C$ with first quasi-homogeneous coordinate non-$0$, there remains to consider only the behavior of $\mathcal S$ and the group action near the other points of $C$, namely $[0:0:1:0]_{\textbf{m}}$ and $[0:0:0:1]_{\textbf{m}}$.  As $T$ permutes these points, we need only look at the action of $\langle S\rangle$ on $\mathcal S_3$ near the first of these.

So, consider the weighted blow-up from inverting $z$.  One has an affine space with coordinates 
$$
z=B_1^{m}, \ x=B_1B_2,\ y=B_1B_3,\ w=B_1^{m}B_4,
$$
divided by the diagonal action of 
$$
\frac{1}{m}[-1,1,1,m].
$$  
The proper transforms of the two equations defining $\mathcal{X}$ yield
$$
B_2B_4+B_3=0,\ B_4+B_2^{2m}+\zeta B_3^{2m}=0,
$$
which define a non-singular surface $\mathcal S_3'$, given in coordinates $B_1,B_2,B_4$ by 
$$
B_4+B_2^{2m}(1+\zeta B_4^{2m})=0,
$$ 
with central curve given by $B_1=0$.
To reach the surface $\mathcal S_3$, one divides by the diagonal action, giving invariants $B_1^m =z$, $B_1B_2=x$, $B_2^m\equiv M$, and $B_4$, now satisfying 
$$
zM=x^m,\ B_4+M^2(1+\zeta B_4^{2m})=0,
$$ 
with central curve $z=x=0$.
Thus, $\mathcal S_3$ has a singularity of type $A_{m-1}$ at the origin, corresponding to the point $[0:0:1:0]_{\textbf{m}}$ of $C$.  
We conclude that the Seifert resolution $\mathcal S$ of the singularity consists of a central hyperelliptic curve of genus $m$ along which are two $A_{m-1}$ singularities. 

Note $S$ lifts to an action on $\mathcal S_3$, calculated to be
$$
S(z,x,M,B_4)=(\omega^{-(2m+1)}z,\ \omega^{-1}x,\ \omega^{m+1}M,\ \omega^{2m+2}B_4).
$$
To complete the description of $\mathcal S/G$, it suffices to consider the quotient of $\mathcal S_3$ by $\langle S\rangle$ at the singular point.

For this purpose, it is easier to extend $S$ to the $m$-fold cover $\mathcal S_3'$, which is smooth.  In the relevant coordinates $B_1, B_2$, and $B_4$
define 
$$
\bar{S}=\frac{1}{mN}[-(2m+1),\  m+1, \ 2m(m+1)].
$$
This extension has the property that $\bar{S}^N=\frac{1}{m}[-1,1,0]$, so dividing gives $\mathcal S_3$; 
and $\bar{S}^m=\frac{1}{N}[-(2m+1), m+1,0]$ gives the same action above of $S$ on the coordinates $z=B_1^m, x=B_1B_2, M=B_2^m,$ and $B_4$ of $\mathcal S_3$. 

So, it suffices to decipher the action of $\bar{S}$ on $\mathcal S_3'$.  Now, 
$$
\bar{S}^{2m^2}=\frac{1}{m+1}[1,0,0]
$$ 
is a pseudoreflection, multiplying $B_1$ by $\omega^{2m}$ and fixing $B_2$ and $B_4$.  Dividing by the group it generates and letting  $\bar{B_1}=B_1^{m+1}$, one has coordinates $\bar{B_1}, B_2, B_4$, the same equation as before (with $\bar{B_1}=0$ the central curve), and group action 
$$
\bar{S}=\frac{1}{2m^2}[-(2m+1),\ 1,\ 2m]
$$ 
on a smooth surface.  This action is free except that $\bar{S}^m$ has fixed points when the first two coordinates are $0$ (and hence so is the last).  Local coordinates at this point are $\bar{B_1}$ and $B_2$, and the group action is $\frac{1}{2m^2}[-(2m+1),\ 1]$.  To describe the resolution of this cyclic quotient singularity on $\mathcal S/G$ in relation to the curve which is the image of $\bar{B_1}=0$, we apply the following well-known result.  

\begin{lemma}[{\cite[pp.~9,10]{laufer-book}}] Consider the cyclic action $(x,y)\mapsto (\mu x, \mu^q y)$ on $\C^2$, where $\mu$ is a primitive $n^{th}$ root of $1$,  $0<q<n$, $(q,n)=1$.  Write the continued fraction expansion
\[
\frac{n}{q}=a_1-\cfrac{1}{a_2-\cfrac{1}{\ddots\ -\cfrac{1}{a_{s-1}-\cfrac{1}{a_s}}}}.
\]
Then the minimal resolution graph of the quotient is 
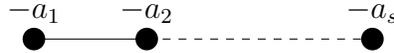
\begin{center}
\begin{tikzpicture}[scale=1.5]
\foreach \a in {0,...,3}
{
\coordinate (A\a) at (\a,0);
}
\draw (A0)--(A1);
\draw[dashed] (A1)--(A3);
\foreach \a in {0,1,3}
{
\fill (A\a) circle [radius=.1];
}
\node[above=2pt] at (A0) {$-a_1$};
\node[above=2pt] at (A1) {$-a_2$};
\node[above=2pt] at (A3) {$-a_s$};
\end{tikzpicture}
\end{center}
and the curve on right intersects transversally the proper transform of the image of $x=0$.
\end{lemma}

Combining what has already been proved about $\mathcal S/G$, one knows the resolution graph of the minimal resolution of $Y/G$, except for the self-intersection $-d$ of the central curve.  The standard calculation of the order of the discriminant group involves $d$ and the outward continued fraction expansions from the center; in this case, it is 
$$
4\cdot 4\cdot 2m^2\left[d-\frac{1}{4}-\frac{1}{4}-\frac{2m^2-2m-1}{2m^2}\right].
$$  
Comparing with the known discriminant value of $16(m+1)^2$, one concludes that $d=2$.  One therefore has the graph $\mathcal B_2^3(m-2).$

\begin{remark}  Assuming the Bhupal-Stipsicz Theorem,  one can by process of elimination conclude that $Y/G$ has type $\mathcal B_2^3$.  For, excluding types $\mathcal C^3$ and $\mathcal B^3$, the only example with non-abelian group and quotient smoothing of a complete intersection singularity is type $\mathcal B^4(p)$.  However, its fundamental group is metacyclic, and the commutator subgroup has index and order incompatible with our $G$.   To rule out $\mathcal C^3$, either use the other theorems in this paper; or, note that $Y/G$ has a graded involution of order $2$ (using a ``square root'' of $S$), whereas neither $\mathcal C^3$ type could have such a symmetry.
\end{remark}

\section{Zariski-van Kampen computations}

\subsection{Fundamental group of the complement of a line arrangement}\label{sec:arrangement}
\mbox{}

Let $\mathcal{L}$ be the projective line arrangement of $\mathbb{P}^2$ in Figure~\ref{fig:arrangement}
(the line at infinity is not part of $\mathcal{L}$), which is the curve~$D$ of~\cite{jf}
for $\mathcal{B}_2^3(p)$. We are going to use the Zariski-van Kampen method
to compute $\pi_1(\mathbb{P}^2\setminus\mathcal{L})$ together with precise descriptions
of meridians close to the singular points, in order to be able to find meridians
of the exceptional components of successive blowing-ups.

Let us denote with small letters 
the standard meridians in the vertical dotted line of Figure~\ref{fig:arrangement}. Let $\mathcal{A}:=\{\ell_1,\ell_2,\ell_3,\ell_4,a_1,a_2,a_3\}$.
We denote the multiple points as $P_{ij}:=L_i\cap L_j$ and $R_i:=A_j\cap A_k$
(where $\{i,j,k\}=\{1,2,3\}$).

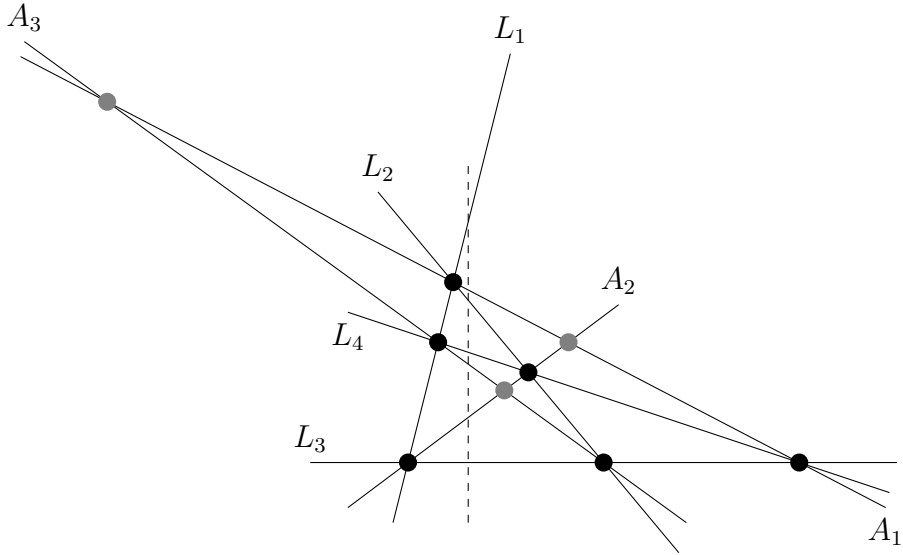
\begin{figure}[ht]
\centering
\begin{tikzpicture}[scale=.8]
\def\ai{.5}
\def\bi{2.4}
\def\aii{1}
\def\bii{2}
\def\aiii{.5}
\def\biii{1.5}
\def\aiv{1}
\def\biv{4}
\coordinate (P13) at (-1.5,-1);
\coordinate (P14) at (-1,1);
\coordinate (P23) at (1.75,-1);
\coordinate (P24) at (.5,.5);

\draw[name path=L1] (${1+\ai}*(P13)-\ai*(P14)$) -- (${1+\bi}*(P14)-\bi*(P13)$)
node[above] {$L_1$};
\draw[name path=L2] (${1+\aii}*(P23)-\aii*(P24)$) -- (${1+\bii}*(P24)-\bii*(P23)$) node[above] {$L_2$};
\draw[name path=L3] (${1+\aiii}*(P13)-\aiii*(P23)$) -- (${1+\biii}*(P23)-\biii*(P13)$)node[above, pos=0] {$L_3$};
\draw[name path=L4] (${1+\aiv}*(P14)-\aiv*(P24)$) -- (${1+\biv}*(P24)-\biv*(P14)$) node[pos=0,below] {$L_4$};

\path [name intersections={of=L1 and L2,by=P12}];
\path [name intersections={of=L3 and L4,by=P34}];

\fill[] (P13) circle [radius=.15cm];
\fill[] (P14) circle [radius=.15cm];
\fill[] (P23) circle [radius=.15cm];
\fill[] (P24) circle [radius=.15cm];
\fill[] (P12) circle [radius=.15cm];
\fill[] (P34) circle [radius=.15cm];

\def\ci{1.25}
\def\di{.25}
\def\cii{.5}
\def\dii{.75}
\def\ciii{2.5}
\def\diii{.5}

\draw[name path=A1] (${1+\ci}*(P12)-\ci*(P34)$) -- (${1+\di}*(P34)-\di*(P12)$) node[below] {$A_1$};
\draw[name path=A2] (${1+\cii}*(P13)-\cii*(P24)$) -- (${1+\dii}*(P24)-\dii*(P13)$) node[above] {$A_2$};
\draw[name path=A3] (${1+\ciii}*(P14)-\ciii*(P23)$) -- (${1+\diii}*(P23)-\diii*(P14)$) node[above,pos=0] {$A_3$};

\path [name intersections={of=A1 and A2,by=Q3}];
\fill[gray] (Q3) circle [radius=.15cm];

\path [name intersections={of=A1 and A3,by=Q2}];
\fill[gray] (Q2) circle [radius=.15cm];

\path [name intersections={of=A2 and A3,by=Q1}];
\fill[gray] (Q1) circle [radius=.15cm];

\def\e{-.5}
\draw[dashed] (\e,-2)--(\e,4);
\end{tikzpicture}
\caption{$7$-line arrangement}
\label{fig:arrangement}
\end{figure}

Let us clarify what we mean by \emph{standard meridians}. In a punctured plane $\mathbb{C}\setminus\Delta$, where
$\Delta$ is an ordered finite set $\{z_1,\dots,z_r\}$, a geometric basis of of $\pi_1(\mathbb{C}\setminus\Delta;*)$, where $z_0\in\mathbb{R}$, $z_0\gg0$,
is a family of meridians as in Figure~\ref{fig:basis} where the product $z_r\cdot\ldots\cdot z_1$ is homotopic to the counterclockwise boundary of a big disk.

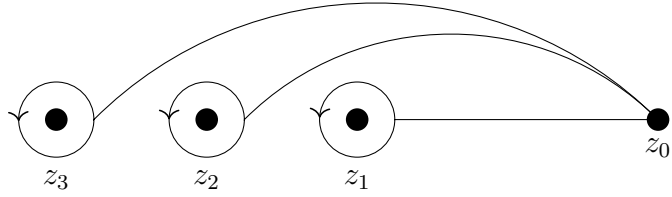
\begin{figure}[ht]
\centering
\begin{tikzpicture}[]
\tikzset{->-/.style={decoration={
  markings,
  mark=at position #1 with {\arrow[scale=1.5]{>}}},postaction={decorate}}}
\fill (5,0) circle [radius=.15] node[below=3pt] {$z_0$};
\fill (1,0) circle [radius=.15] node[below=15pt] {$z_1$};
\draw[->-=.76] (5,0) -- (1.5,0) arc [start angle=0,end angle=360, radius=.5];
\fill (-1,0) circle [radius=.15] node[below=15pt] {$z_2$};
\draw[->-=.83] (5,0) to[out=135, in=45] (-.5,0) arc [start angle=0,end angle=360, radius=.5];
\fill (-3,0) circle [radius=.15] node[below=15pt] {$z_3$};
\draw[->-=.865] (5,0) to[out=135, in=45] (-2.5,0) arc [start angle=0,end angle=360, radius=.5];
\end{tikzpicture}
\caption{Geometric basis, $r=3$}
\label{fig:basis}
\end{figure}

We follow the classical method of Zariski-van Kampen. Let us fix a base point $p:=(x_0,y_0)$, where $x_0$ is the coordinate of the dotted line in Figure~\ref{fig:arrangement},
where the standard meridians lie. We know that $\mathcal{A}$ generates $G:=\pi_1(\mathbb{P}^2\setminus\mathcal{L};p)$, and since the line at infinity is not part of $\mathcal{L}$, the following relation holds:
\begin{equation}\label{infinito}\tag{$\infty$}
1=\ell_3\cdot a_2\cdot a_3\cdot\ell_4\cdot\ell_2\cdot a_1\cdot\ell_1
\end{equation}
The rest of the relations will be found using the Zariski-van Kampen method.  In order to find them we need to enlarge the concept of standard meridian. In a vertical line $x=x_1$ we can consider also a geometric basis with base point $(x_1,y_0)$. These new meridians
can be seen as elements of $G$ if we conjugate them by a path joining $p$ and $(x_1,y_0)$ in the horizontal line $y=y_0$ which avoids 
the $x$-coordinates of the multiple points of $\mathcal{L}$. These new meridians can be written in terms of the original ones, and this is
usually the difficult part of  the Zariski-van Kampen method. The real picture in Figure~\ref{fig:arrangement} provides all the needed
information.

Let us recall how this works for double and triple points using Figure~\ref{fig:doubletriple}. In both cases the standard generators
to the left of the multiple point are expressed in terms of the generators to the right, taking into account 
the relations created by the singular point $[a,b]=1$ for a double point and $c\cdot b\cdot a=b\cdot a\cdot c =a\cdot c\cdot b$
for a triple point;
we will denote the second relation as $[c,b,a]=1$. Note that the element $e$ commutes with the involved meridians in each case. 
If we call $E$ the exceptional component of the blowing-up of the multiple point, then $e$ is a meridian of this component.

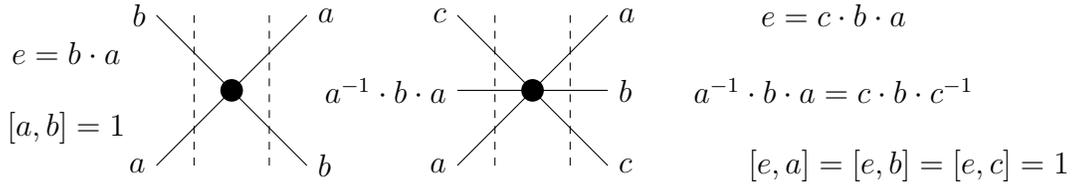
\begin{figure}[ht]
\centering
\begin{tikzpicture}
\draw (-1,-1) node[left] {$a$} --(1,1) node[right] {$a$};
\draw (-1,1) node[left] {$b$} --(1,-1) node[right] {$b$};
\fill (0,0) circle [radius=.15cm];
\draw[dashed] (-.5,-1)--(-.5,1);
\draw[dashed] (.5,-1)--(.5,1);
\node at (-2.2,.5) {$e=b\cdot a$};
\node at (-2.2,-.5) {$[a,b]=1$};

\begin{scope}[xshift=4cm]
\draw (-1,-1) node[left] {$a$} --(1,1) node[right] {$a$};
\draw (-1,0) node[left] {$a^{-1}\cdot b\cdot a$} --(1,0) node[right] {$b$};
\draw (-1,1) node[left] {$c$} --(1,-1) node[right] {$c$};
\fill (0,0) circle [radius=.15cm];
\draw[dashed] (-.5,-1)--(-.5,1);
\draw[dashed] (.5,-1)--(.5,1);
\node at (4,1) {$e=c\cdot b\cdot a$};
\node at (4,0) {$a^{-1}\cdot b\cdot a=c\cdot b\cdot c^{-1}$};
\node at (5,-1) {$[e,a]=[e,b]=[e,c]=1$};
\end{scope}
\end{tikzpicture}
\caption{Local picture at double and triple points}
\label{fig:doubletriple}
\end{figure}

Let us consider the multiple points to the left of the dotted vertical line, see Figure~\ref{fig:left}. We obtain
the following relations:
\begin{align}
\label{b12}\tag{$\rho_{12}$}
[\ell_2,a_1,\ell_1]&=1\\
\label{b14}\tag{$\rho_{14}$}
[a_3,\ell_4,\ell_1]&=1\\
\label{b13}\tag{$\rho_{13}$}
[\ell_3,a_2,\ell_1]&=1\\
\label{a2}\tag{$\rho_{2}$}
[a_3,\ell_2\cdot a_1\cdot\ell_2^{-1}]&=1
\end{align}

\begin{figure}[ht]
\centering
\begin{tikzpicture}[scale=.5]
\draw (-1,-1) node[left] {$\ell_1$} --(1,1) node[right] {$\ell_1$};
\draw (-1,0) node[left] {$\scriptstyle\ell_2\cdot a_1\cdot\ell_2^{-1}$} --(1,0) node[right] {$a_1$};
\draw (-1,1) node[left] {$\ell_2$} --(1,-1) node[right] {$\ell_2$};
\fill (0,0) circle [radius=.15cm];
\draw[dashed] (-.5,-1)--(-.5,1);
\draw[dashed] (.5,-1)--(.5,1);
\node at (0,-2) {$P_{12}$};
\begin{scope}[xshift=6cm]
\draw (-1,-1) node[left] {$\ell_1$} --(1,1) node[right] {$\ell_1$};
\draw (-1,0) node[left] {$\scriptstyle a_3\cdot\ell_4\cdot a_3^{-1}$} --(1,0) node[right] {$\ell_4$};
\draw (-1,1) node[left] {$a_3$} --(1,-1) node[right] {$a_3$};
\fill (0,0) circle [radius=.15cm];
\draw[dashed] (-.5,-1)--(-.5,1);
\draw[dashed] (.5,-1)--(.5,1);
\node at (0,-2) {$P_{14}$};
\end{scope}
\begin{scope}[xshift=12cm]
\draw (-1,-1) node[left] {$\ell_1$} --(1,1) node[right] {$\ell_1$};
\draw (-1,0) node[left] {$\scriptstyle\ell_3\cdot a_2\cdot \ell_3^{-1}$} --(1,0) node[right] {$a_2$};
\draw (-1,1) node[left] {$\ell_3$} --(1,-1) node[right] {$\ell_3$};
\fill (0,0) circle [radius=.15cm];
\draw[dashed] (-.5,-1)--(-.5,1);
\draw[dashed] (.5,-1)--(.5,1);
\node at (0,-2) {$P_{13}$};
\end{scope}
\begin{scope}[xshift=16cm]
\draw (-1,-1)  --(1,1) node[right] {$\scriptstyle\ell_2\cdot a_1\cdot\ell_2^{-1}$};
\draw (-1,1)  --(1,-1) node[right] {$a_3$};
\fill (0,0) circle [radius=.15cm];
\draw[dashed] (.5,-1)--(.5,1);
\node at (0,-2) {$R_2$};
\end{scope}
\end{tikzpicture}
\caption{Multiple points to the left of the dotted vertical line}
\label{fig:left}
\end{figure}
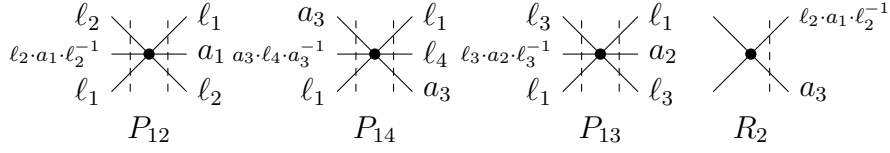

Let us continue with the multiple points to the right of the dotted line, see Figure~\ref{fig:right}. Since we always 
avoid the singular fiber running counterclockwise, the situation
of Figure~\ref{fig:doubletriple} applies, interchanging left and right.

\begin{figure}[ht]
\centering
\begin{tikzpicture}[scale=.5]
\draw (-1,-1) node[left]  {$a_2$} --(1,1) node[right] {$a_2$};
\draw (-1,1) node[left] {$a_3$} --(1,-1) node[right] {$a_3$};
\fill (0,0) circle [radius=.15cm];
\draw[dashed] (-.5,-1)--(-.5,1);
\draw[dashed] (.5,-1)--(.5,1);
\node at (0,-2) {$R_1$};
\begin{scope}[xshift=5.cm]
\draw (-1,-1) node[left] {$a_2$} --(1,1) node[right] {$a_2$};
\draw (-1,0) node[left] {$\ell_4$} --(1,0) node[right] {$\scriptstyle a_2\cdot\ell_4\cdot a_2^{-1}$};
\draw (-1,1) node[left] {$\ell_2$} --(1,-1) node[right] {$\ell_2$};
\fill (0,0) circle [radius=.15cm];
\draw[dashed] (-.5,-1)--(-.5,1);
\draw[dashed] (.5,-1)--(.5,1);
\node at (0,-2) {$P_{24}$};
\end{scope}
\begin{scope}[xshift=10cm]
\draw (-1,-1) node[left]  {$a_2$} --(1,1) node[right] {$a_2$};
\draw (-1,1) node[left] {$a_1$} --(1,-1) node[right] {$a_1$};
\fill (0,0) circle [radius=.15cm];
\draw[dashed] (-.5,-1)--(-.5,1);
\draw[dashed] (.5,-1)--(.5,1);
\node at (0,-2) {$R_3$};
\end{scope}
\begin{scope}[xshift=14cm]
\draw (-1,-1) node[left] {$\ell_3$} --(1,1) node[right] {$\ell_3$};
\draw (-1,0) node[left] {$a_3$} --(1,0) node[right] {$\scriptstyle\ell_3\cdot a_2\cdot\ell_3^{-1}$};
\draw (-1,1) node[left] {$\ell_2$} --(1,-1) node[right] {$\ell_2$};
\fill (0,0) circle [radius=.15cm];
\draw[dashed] (-.5,-1)--(-.5,1);
\draw[dashed] (.5,-1)--(.5,1);
\node at (0,-2) {$P_{23}$};
\end{scope}
\begin{scope}[xshift=22cm]
\draw (-1,-1) node[left] {$\ell_3$} --(1,1);
\draw (-1,0) node[left] {$\scriptstyle a_2\cdot\ell_4\cdot a_2^{-1}$} --(1,0) ;
\draw (-1,1) node[left] {$a_1$} --(1,-1) ;
\fill (0,0) circle [radius=.15cm];
\draw[dashed] (-.5,-1)--(-.5,1);
\node at (0,-2) {$P_{34}$};
\end{scope}
\end{tikzpicture}
\caption{Multiple points to the right of the dotted vertical line}
\label{fig:right}
\end{figure}
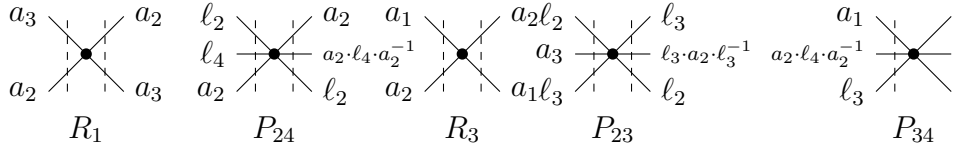

We obtain the following relations:
\begin{align}
\label{a1}\tag{$\rho_1$}
[a_2,a_3]&=1\\
\label{b24}\tag{$\rho_{24}$}
[a_2,\ell_4,\ell_2]&=1\\
\label{a3}\tag{$\rho_3$}
[a_1,a_2]&=1\\
\label{b23}\tag{$\rho_{23}$}
[\ell_3,a_3,\ell_2]&=1\\
\label{b34}\tag{$\rho_{34}$}
[\ell_3,a_2\cdot\ell_4\cdot a_2^{-1},a_1]&=1
\end{align}

Let us summarize these computations.

\begin{proposition}
The group $G=\pi_1(\mathbb{P}^2\setminus\mathcal{L})$ is generated by $\mathcal{A}$ with relations
\eqref{infinito}, \eqref{b12}, \eqref{b13}, \eqref{b14}, \eqref{b23}, \eqref{b24},
\eqref{b34}, \eqref{a1}, \eqref{a2}, and \eqref{a3}.
\end{proposition}

This is a non-abelian group, as it is the case for any line arrangement having
\emph{more-than-double} points. In fact, this presentation has redundant
generators and relations but it will be useful to express the meridians of the
exceptional components of the blow-ups. If we blow up a point $P_{ij}$, we denote by $E_{ij}$
its exceptional component and by $e_{ij}$ a suitable meridian.
In the same way, if we blow up a point $R_{i}$, we denote by $E_{i}$
its exceptional component and by $e_{i}$ a suitable meridian.

\begin{lemma}
The meridians $e_{ij}, e_{i}$ have the following expressions in terms
of the generators of~$G$:
\begin{align*}
e_{12}&:=\ell_2\cdot a_1\cdot\ell_1&
e_{13}&:=\ell_3\cdot a_2\cdot\ell_1\\
e_{14}&:=\ell_1\cdot a_3\cdot\ell_4&
e_{23}&:=\ell_3\cdot a_3\cdot\ell_2\\
e_{24}&:=\ell_2\cdot a_2\cdot\ell_4&
e_{34}&:=a_2\cdot\ell_4\cdot a_2^{-1}\cdot a_1\cdot\ell_3\\
e_1&:=a_2\cdot a_3&
e_2&:=a_3\cdot\ell_2\cdot a_1\cdot\ell_2^{-1}\\
e_3&:=a_1\cdot a_2.
\end{align*}
\end{lemma}

\subsection{Fundamental group of the complement of the Milnor fiber of a smoothing of \texorpdfstring{$\mathcal{B}^3_2(p)$}{B23(p)}}\label{sec:b23}
\mbox{}

In this subsection we perform the direct computation of the fundamental group of the Milnor fiber of a smoothing of $\mathcal{B}^3_2(p)$;
the result coincides with the group described in \S\ref{sec:group}.

Let $\pi:X^2_3\to\mathbb{P}^2$ the composition of the following blowing-ups. First we 
blow-up all the points $P_{ij}$. 
In this intermediate surface, we blow up
$Q_1$ (resp. $Q_2$), the intersection point of $E_{23}$ (resp. $E_{14}$) and the strict transform of $A_3$.
Finally we perform $p+1$ extra blowing-ups
over $R_3$, all of them on the intersection point of the strict transform of $A_1$
and the previous exceptional component.
It is obvious that $X^2_3\setminus\pi^{-1}(\mathcal{L})$ is isomorphic to $\mathbb{P}^2\setminus\mathcal{L}$.
The curve $\pi^{-1}(\mathcal{L})$ has $16+p$ connected components (see Figure~\ref{fig:graphB}).

Let $\mathcal{B}\subset\pi^{-1}(\mathcal{L})$ the curve obtained as union of the strict transforms
of $L_j$, $A_i$ and $E_{14},E_{23},E_{12}$ and the strict transforms of the first  $p$
exceptional components over $R_3$, i.e, $10+p$ connected components.
It is a normal crossing divisor whose dual graph is shown in Figure~\ref{fig:graphB}.

\begin{figure}[ht]
\centering
\begin{tikzpicture}
\foreach \a in {0,...,6}
{
\coordinate (A\a) at (\a,0);
\fill (A\a) circle [radius=0.1];
}
\foreach \b in {1,...,3}
{
\coordinate (B\b) at (3,\b);
\fill (B\b) circle [radius=0.1];
}
\coordinate (Q1) at (6,1);
\node[above] at (Q1) {$Q_1$};
\fill[gray] (Q1) circle [radius=0.1cm];
\draw[dotted] (B2)--(Q1)--(A5);
\coordinate (Q2) at (0,1);
\node[above] at (Q2) {$Q_2$};
\fill[gray] (Q2) circle [radius=0.1cm];
\draw[dotted] (B2)--(Q2)--(A1);
\draw (A0)--(A6);
\draw (A3)--(B3);
\node[below=3pt] at (A0) {$L_4$};
\node[below=3pt] at (A1) {$E_{14}$};
\node[below=3pt] at (A2) {$L_{1}$};
\node[below=3pt] at (A3) {$E_{12}$};
\node[above left=0pt] at (A3) {$-1$};
\node[below=3pt] at (A4) {$L_2$};
\node[below=3pt] at (A5) {$E_{23}$};
\node[below=3pt] at (A6) {$L_{3}$};
\node[above right=1pt] at (B1) {$A_{1}$};
\node[left=3pt] at (B1) {$-(p+2)$};
\node[right=3pt] at (B2) {$A_{3}$};
\node[above right=3pt] at (B3) {$A_{2}$};
\node[left=3pt] at (B2) {$-3$};
\node[below=3pt] at (2,3) {$E_{3,1}$};
\node[below=3pt] at (0,3) {$E_{3,p}$};
\fill (2,3) circle [radius=0.1];
\fill (0,3) circle [radius=0.1];
\draw (2,3)--(B3);
\draw[dashed] (2,3)--(0,3);
\coordinate (D) at ($.4*(B1)+.6*(0,3)$);
\draw[dotted] (B1)--(0,3);
\fill[gray] (D) circle [radius=0.1cm];
\node[below] at (D) {$E_{3,p+1}$};

\coordinate (E13) at (4,-1.5);
\fill[gray] (E13) circle [radius=0.1];
\node[below right] at (E13) {$E_{13}$};
\draw[dotted] (A2)--(E13);
\draw[dotted] (A6)--(E13);
\draw[dotted] (B3) to[out=0, in=90] (7,1) to[out=-90, in=0] (E13) ;

\coordinate (E24) at (2,-1.5);
\fill [gray] (E24) circle [radius=0.1];
\node[below right] at (E24) {$E_{24}$};
\draw[dotted] (A4)--(E24);
\draw[dotted] (A0)--(E24);
\draw[dotted] (B3) to[out=90, in=90] (-2,1) to[out=-90, in=180] (E24) ;

\coordinate (E34) at (4.5,1);
\fill[gray] (E34) circle [radius=0.1];
\node[below] at (E34) {$E_{34}$};
\draw[dotted] (B1)--(E34);
\draw[dotted] (A0) --(E34);
\draw[dotted] (A6) -- (E34) ;

\end{tikzpicture}
\caption{Missing self-intersections are $-2$ if black and $-1$ if gray.}
\label{fig:graphB}
\end{figure}
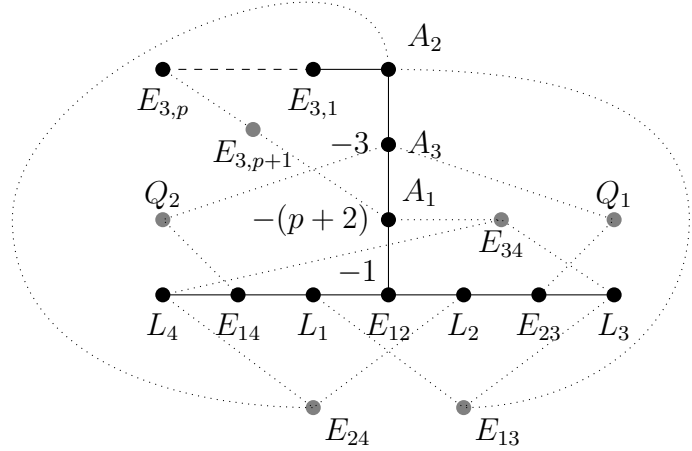

The missing curves are $E_{13},E_{24},E_{34}$,
the last exceptional component over
$R_3$, and the exceptional divisors coming from $Q_1,Q_2$;
the meridians of those ones are
\begin{align*}
e_{3,j}&:=a_1^{j}\cdot a_2\\
q_1&:=e_{23}\cdot a_3=\ell_2\cdot\ell_3\cdot a_3^2\\
q_2&:=e_{14}\cdot a_3=\ell_4\cdot\ell_1\cdot a_3^2.
\end{align*}
In order to compute $G_1:=\pi_1(X_3^2\setminus\mathcal{B})$
we need the following classical result.

\begin{proposition}[{\cite[Lemma~4.18]{fujita:82}}]\label{prop:fujita}
Let $X$ be a smooth complex projective surface, $D\subset X$ a reduced divisor, and $D'=D\cup A_1 \cup\cdots \cup A_r$, where 
$A_1,\dots,A_r$ are the irreducible components of $D'$ which are \emph{not} in $D$.

Let $i_*:\pi_1(X\setminus D')\to\pi_1(X\setminus D)$. Then $i_*$ is surjective and $\ker i_*$
is generated by all the meridians of the components $A_i$ in $X$.
\end{proposition}

This statement is sharper than Fujita's original one, where $r=1$ and $A_1$ must
be transversal to~$D$; the combination of induction and embedded resolution of 
$D$ reduces this statement to the original one.

Therefore, $G_1$ is isomorphic to the quotient 
of $G$ by the normal subgroup generated by $e_{13}$, $e_{24}$, $e_{34}$, $e_{3,p+1}$,
$q_1$, and $q_2$.
A presentation of $G_1$ is obtained from the presentation of $G$ by adding the 
relations coming from killing the above meridians; we can forget the relations 
\eqref{b12}, \eqref{b24}, \eqref{b34} and \eqref{a3}.
Summarizing, $G_1$ is generated by $\mathcal{A}$, i.e., $\ell_1,\dots,\ell_4,a_1,a_2,a_3$, with
the following relations:
\begin{align}\notag
\eqref{a1}&&
[a_2,a_3]&=1\\
\label{e24}\tag{$\sigma_{24}$}&&
a_2\cdot\ell_4\cdot\ell_2&=1\\
\label{e13}\tag{$\sigma_{13}$}&&
\ell_3\cdot a_2\cdot \ell_1&=1\\
\label{e34}\tag{$\sigma_{34}$}&&
\ell_3\cdot a_2\cdot\ell_4\cdot a_2^{-1}\cdot a_1&=1\\
\label{e3}\tag{$\sigma_3$}&&
a_1^{p+1}\cdot a_2&=1\\
\notag
\eqref{infinito}&&
\ell_3\cdot a_2\cdot a_3\cdot\ell_4\cdot\ell_2\cdot a_1\cdot\ell_1&=1\Longleftrightarrow
a_3\cdot a_1=a_2\\
\label{q1}\tag{$\tau_1$}&&
\ell_2\cdot\ell_3\cdot a_3^2&=1\\
\label{q2}\tag{$\tau_2$}&&
\ell_4\cdot\ell_1\cdot a_3^2&=1\\
\notag\eqref{b12}&&
[\ell_2,a_1,\ell_1]&=1\\
\notag\eqref{b14}&&
[a_3,\ell_4,\ell_1]&=1\Leftrightarrow
\ell_1\cdot a_3\cdot\ell_4\cdot a_3=1
\Leftrightarrow[\ell_1,a_3]\!=\![\ell_4,a_3]=1\\
\notag\eqref{a2}&&
[a_3,\ell_2\cdot a_1\cdot\ell_2^{-1}]&=1\\
\notag\eqref{b23}&&
[\ell_3,a_3,\ell_2]&=1
\Leftrightarrow[\ell_2,a_3]\!=\![\ell_3,a_3]=1
\end{align}
This implies that $a_3$ is central, which replaces \eqref{b23}, \eqref{a2}, \eqref{b14}, and \eqref{a1}.
Some generators can be eliminated:
\begin{align*}
a_2&=a_1^{-(p+1)}&
a_3&=a_1^{-(p+2)}&
\ell_2&=\ell_1\cdot a_1^{-(p+3)}\\
\ell_3&=\ell_1^{-1}\cdot a_1^{p+1}&
\ell_4&=a_1^{2(p+2)}\cdot\ell_1^{-1}
\end{align*}
Hence the group is generated by $a_1,\ell_1$ with the following relations
\begin{align*}
[a_1^{p+2},\ell_1]&=1&
[a_1,\ell_1^2]&=1\\
\ell_1\cdot a_1\cdot\ell_1^{-1}\cdot a_1^{2p+5}&=1&
\ell_1\cdot a_1\cdot\ell_1 &=a_1^{p+1}
\end{align*}
As a consequence $a_1^{2(p+2)(p+3)}=1$, and then 
$\ell_1^2\cdot a_1^{p+6}=1$.
Hence, calling $a:=a_1$, $\ell:=\ell_1$, and $q:=p+3$ we have:
\begin{equation}\label{eq:presentation}
G_1=\langle a,\ell\mid
a^{2(q-1)q}=1,\ell^2= a^{3(q-1)},\ell\cdot a\cdot\ell^{-1}= a^{1-2 q}
\rangle.
\end{equation}
Note that $\ell^2,a^{(q-1)}$ are central. 
The group fits
in a short exact sequence
\[
1\to C_{2q(q-1)}\to G_1\to C_2\to 1
\]
where $C_{j}$ is a cyclic group of order $j$. Any element of 
$G_1$ admits a unique representation of the form
$\ell^{\varepsilon}\cdot a^{m}$ where 
$\varepsilon\in\{0,1\}$ and $m\in\{0,1,\dots,2(q-1)q-1\}$:
\begin{align*}
a_2&=a_1^{2-q}&
a_3&=a_1^{1-q}&
\ell_2&=\ell_1\cdot a_1^{-q}\\
\ell_3&=\ell_1\cdot a_1^{1-2q}&
\ell_4&=\ell_1\cdot a_1^{1-q}&
e_{12}&=a_1^{2(q-1)}\\
e_{14}&=a^{q-1}&
e_{23}&=a^{q-1}
\end{align*}

Let us compare this presentation with the one given in \S\ref{sec:group}. In the notation $m=p+2=q-1$.
The element $S$ in \S\ref{sec:group} corresponds with  the element~$a$ in~\eqref{eq:presentation} 
while $T$ corresponds with $\ell a$.

\subsection{Fundamental group of the complement of the Milnor fiber of a smoothing of \texorpdfstring{$\mathcal{C}^2_3(p)$}{C23(p)}}
\label{sec:c23}
\mbox{}

We are going to study a projective curve $\mathcal{C}_2\cup\mathcal{C}_3\cup\mathcal{T}_\infty$, where $\mathcal{C}_3$ is a nodal cubic with
node at $P\in\mathbb{P}^2$, $\mathcal{C}_2$ is a smooth conic, and $\mathcal{T}_\infty$  is a line satisfying:
\begin{itemize}
\item  $\mathcal{C}_2\cap\mathcal{C}_3=\{Q\}$, where $P\neq Q$ (from Bézout's theorem $(\mathcal{C}_2\cdot\mathcal{C}_3)_{Q}=6$).
\item $\mathcal{T}_\infty$ is one of the tangent lines to $\mathcal{C}_3$ at $P$.
\end{itemize}
It is not hard to see that there is only such a curve
up to projective transformation.
Equations can be given:
\begin{align*}
\mathcal{C}_3&:y^2 z=x^2(x+z)\\
\mathcal{C}_2&:y^2=-(x+z)(2x+z)\\
\mathcal{T}_\infty&:y=x.
\end{align*}
The other tangent line to $\mathcal{C}_3$ at $P$ is denoted by $\mathcal{T}_0$ and its equation is $y+x=0$.

\begin{remark}
Fowler showed that $\mathcal{C}^2_3(p)$ had two distinct smoothing components (related by complex conjugation) which seems to be
in contradiction with the projective rigidity of $\mathcal{C}_2\cup\mathcal{C}_3\cup\mathcal{T}_\infty$. In fact,
there is no such contradiction; note that as shown in Figure~\ref{fig:c23} for the construction of the Milnor fiber
we need to perform some blow-ups at one of the two points of $\mathcal{C}_2\cap\mathcal{T}_\infty$ 
(which are complex conjugate with the above equations!), and this fact confirms the existence of two distinct smoothing components.
\end{remark}

Unfortunately, the real picture in Figure~\ref{fig:c3c2tinfty} does not contain all the topological information of the curve,
mainly due the fact that $\mathcal{C}_2$ and $\mathcal{T}_\infty$ do not intersect at real points.

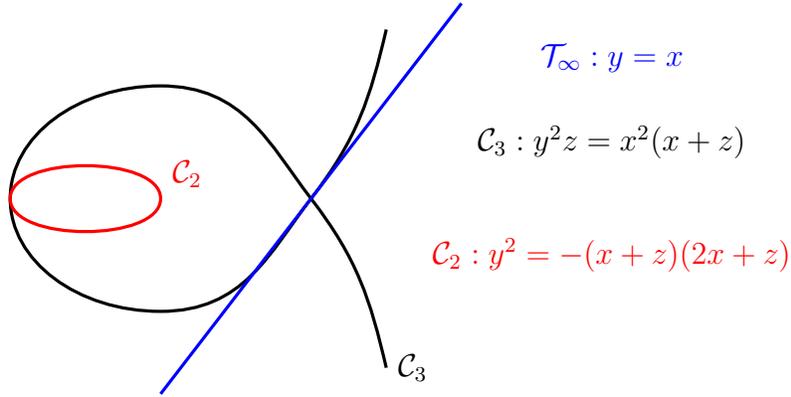
\begin{figure}[ht]
\begin{tikzpicture}[yscale=.75]
\draw[line width=1.2] (1,3) to[out=-100, in=60] (0,0) to[out=-120,in=0] (-2,-2)  to[out=180, in=-90] (-4,0)
to[out=90,in=180] (-2,2) to[out=0,in=120] (0,0) to[out=-60,in=100] (1,-3) node[right] {$\mathcal{C}_3$};
\draw[line width=1.2,color=red] (-2,0) to[out=90,in=90] (-4,0) to[out=-90,in=-90] (-2,0)node[above right] {$\mathcal{C}_2$};

\draw[line width=1.2,blue] (-120:4) -- (60:4);

\node at (4,1) {$\mathcal{C}_3:y^2 z=x^2(x+z)$};
\node[red] at (4,-1) {$\mathcal{C}_2:y^2=-(x+z)(2x+z)$};
\node[blue] at (4,2.5) {$\mathcal{T}_\infty:y=x$};

\end{tikzpicture}
\caption{Real picture of $\mathcal{C}_2\cup\mathcal{C}_3\cup\mathcal{T}_\infty$.}
\label{fig:c3c2tinfty}
\end{figure}

\begin{theorem}\label{thm:c3c2tinfty}
The fundamental group of the complement of 
$\mathcal{C}_2\cup\mathcal{C}_3\cup\mathcal{T}_\infty$ in $\mathbb{P}^2$
is isomorphic to $\mathbb{Z}\oplus\mathbb{Z}$.
\end{theorem}

Before giving the proof of this theorem, let us  show how the plane curve curve 
$\mathcal{C}_2\cup\mathcal{C}_3\cup\mathcal{T}_\infty$ is related to the Milnor fibre
of the $\qhd$-smoothing of $\mathcal{C}^2_3(p)$.
This curve follows the ideas in~\cite{w7} to find the curve at infinity for~$\mathcal{C}_2^3(p)$, using conics. In~\cite{jf}, the author proceeds using a line arrangement with 9 lines: the McLane arrangement and one line joining two triple points.

\begin{figure}[ht]
\begin{center}
\begin{tikzpicture}
\foreach \x in {-2,...,5}
{
\coordinate (A\x) at (\x,0);
}
\foreach \x in {-2,...,5}
{
\fill (A\x) circle [radius=.1];
}
\foreach \x in {-4,...,0}
{
\coordinate (B\x) at (\x,1);
\fill (B\x) circle [radius=.1];
}
\draw (A0)--(B0)--(B-4);
\coordinate (C) at (-1.5,-1);
\fill[color=gray] (C) circle [radius=.1];
\draw[dotted] (A-1)--(C)--(A-2);
\draw(A-2)--(A3);
\draw[dashed] (A3)--(A4);
\draw (A4)--(A5);
\coordinate (D) at (5,1);
\fill[color=gray] (D) circle [radius=.1];
\draw[dotted] (A2) to[out=-130,in=0] (C);
\draw[dotted] (A5) -- (D)--(A1);

\node[below] at (A0) {$-1$};
\node[right] at (D) {$-1$};
\node[above=2pt] at (A-1) {$\mathcal{C}_3$};
\node[below=2pt] at (A1) {$\mathcal{C}_2$};
\node[above=2pt] at (A1) {$-(p+3)$};
\node[below=2pt] at (A2) {$\mathcal{T}_\infty$};
\node[below=2pt] at (A3) {$E_1$};
\node[below=2pt] at (A4) {$E_{p-1}$};
\node[below=2pt] at (A5) {$E_{p}$};
\node[left] at (C) {$-1$};
\end{tikzpicture}
\caption{Resolution of $\mathcal{C}_2\cup\mathcal{C}_3\cup\mathcal{T}_\infty$  including
$\mathcal{C}_2^3(p)$.}
\label{fig:c23}
\end{center}
\end{figure}
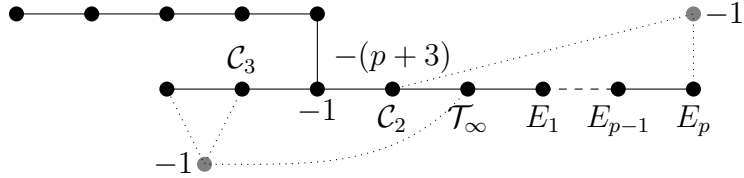

\begin{corollary}\label{cor:c23}
The Milnor fibre of the $\qhd$-smoothing of $\mathcal{C}^2_3(p)$ is abelian.
\end{corollary}

\begin{proof}
In Figure~\ref{fig:c23} we have depicted a (non-minimal) embedded resolution of the singularities
of $\mathcal{C}_2\cup\mathcal{C}_3\cup\mathcal{T}_\infty$. Let $\pi:X\to\mathbb{P}^2$ be that resolution.
Then $\mathbb{P}^2\setminus(\mathcal{C}_2\cup\mathcal{C}_3\cup\mathcal{T}_\infty)$
is isomorphic to $X\setminus\pi^{-1}(\mathcal{C}_2\cup\mathcal{C}_3\cup\mathcal{T}_\infty)$
and then its fundamental group is abelian.

We obtain the Milnor fibre $F$ of the $\qhd$-smoothing of $\mathcal{C}^2_3(p)$ as 
the complement in $X$ of all the irreducible components of $\pi^{-1}(\mathcal{C}_2\cup\mathcal{C}_3\cup\mathcal{T}_\infty)$
with the exception of the \emph{gray} components in Figure~\ref{fig:c23}. Then $\pi_1(F)$
is a quotient of $\mathbb{Z}^2$ by Proposition~\ref{prop:fujita} and the statement follows.
\end{proof}
In Figure~\ref{fig:c23} one can see the dual graph of a resolution of $\mathcal{C}_3\cup\mathcal{C}_2\cup\mathcal{T}_\infty$, with extra blow-ups at one of the points in $\mathcal{C}_2\cap\mathcal{T}_\infty$.
\vspace{3mm} 

Actually, Theorem~\ref{thm:c3c2tinfty} can be proved using \texttt{SIROCCO}~\cite{mbrr:16} inside \texttt{Sagemath}~\cite{sagemath}. A simple explanation
on how it works can be found in~\cite{marco}.
The code is very simple:
\begin{python}
R.<x,y,z>=QQ[]
F=(y^2*z-x^2*(x+z))*(y^2+(x+z)*(2*x+z))*(y-x)
C=Curve(F)
C.fundamental_group()
\end{python}

We include a computer-free proof of the Theorem. The strategy is to apply birational transformations to obtain an arrangement of curves
in $\C^2$ such that the complement of this arrangement is isomorphic
to $\mathbb{P}^2\setminus(\mathcal{C}_2\cup\mathcal{C}_3\cup\mathcal{T}_\infty\cup\mathcal{T}_0)$. This arrangement has real equations, and moreover
the real picture contains all the topological information. We can compute the fundamental group using the Zariski-van Kampen method applied to the vertical
projection. The first interesting property of the curve is that all the \emph{non-transversal} vertical lines are in the real picture. Not all the real
vertical lines intersect the arrangement of curves at real points, but the real part of the intersections can be tracked. As a consequence,
the real picture allows one to find the braid monodromy of the curve, and so the fundamental group can be computed. In order 
to obtain the fundamental group of $\mathbb{P}^2\setminus(\mathcal{C}_2\cup\mathcal{C}_3\cup\mathcal{T}_\infty)$
an extra step is needed. We can compute the meridian of $\mathcal{T}_0$ in terms on the given presentation; it is enough
to \emph{kill} this meridian.

\begin{proof}[Proof of Theorem{\rm~\ref{thm:c3c2tinfty}}]
Let us blow-up the nodal point $[0:0:1]$ of~$\mathcal{C}_3$. Let $E$ be the exceptional component of the resulting ruled surface~$\Sigma_1$, see Figure~\ref{fig:sigma1}. 

\begin{figure}[ht]
\centering
\begin{tikzpicture}
\draw (-2.5,0)--(.5,0);
\draw (0,-2.5)--(0,.5);
\draw (-2,-2.5)--(-2,.5);
\draw (-1,-1) ellipse [x radius=1.5cm,y radius=.5cm];
\draw (-2.5,.5) to[out=-45,in=150] (-2,0)[out=-30,in=180] to (-1,-.5) to[out=0,in=-150] (0,0) to[out=30,in=-120] (.5,.5);
\node[left] at (-2.5,-1) {$\mathcal{C}_2$};
\node[right] at (0,-2) {$\mathcal{T}_0$};
\node[right] at (.5,0) {$E$};
\node[left] at (-2,-2) {$\mathcal{T}_\infty$};
\node[left]at (-2.5,.5) {$\mathcal{C}_3$};
\end{tikzpicture}
\caption{Combinatorial picture in $\Sigma_1$ (the intersections of $\mathcal{C}_2$ with $\mathcal{T}_i$ are not real).}
\label{fig:sigma1}
\end{figure}
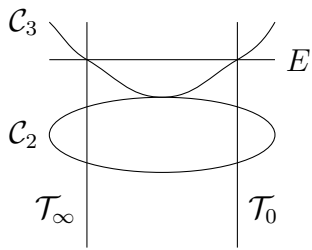

We continue with a couple of elementary transformations which yield $\Sigma_3$. We blow up $E\cap\mathcal{T}_i$ and contract the strict
transforms of $\mathcal{T}_\infty,\mathcal{T}_0$, keeping the exceptional components $\mathcal{F}_\infty,\mathcal{F}_0$. The complement 
of $\mathcal{C}_2\cup\mathcal{C}_3\cup\mathcal{T}_\infty\cup\mathcal{T}_0$ in $\mathbb{P}^2$ is isomorphic to the complement 
of $\mathcal{C}_2\cup\mathcal{C}_3\cup E\cup \mathcal{F}_\infty\cup \mathcal{F}_0$ in $\Sigma_3$. 
We can provide equations. If we blow-down the $(-3)$-section $E$ we obtain the weighted projective plane $\mathbb{P}^2_{\omega}$,
$\omega:=(1,1,3)$,
where the curves are defined by weighted homogeneous polynomials in the variables~$x_\omega,y_\omega,z_\omega$. 
In fact, the birational transformation is
\[
\begin{tikzcd}[row sep=0pt,/tikz/column 1/.append style={anchor=base east},/tikz/column 2/.append style={anchor=base west}]
\mathbb{P}^2\ar[r,dashed,"\psi"]&\mathbb{P}^2_\omega\\
{[x:y:z]}\ar[r,mapsto]&\left[ y-x  : x + y  : 8 (y^2 z -x^2(x+z))\right]_\omega.
\end{tikzcd}
\]
The curves $\mathcal{F}_\infty,\mathcal{F}_0$ have equations
$x_\omega=0$, $y_\omega=0$, respectively, while $\mathcal{C}_3$ has equation $z_\omega=0$.  By writing down the inverse of the map $\psi$, a long but straightforward calculation yields that the equation of $\mathcal{C}_2$ is
\begin{equation}\label{eq:c2}
z_\omega^2-2 (x_\omega^3+3 x_\omega^2 y_\omega-3 x_\omega y_\omega^2-y_\omega^3) z_\omega+(x_\omega+y_\omega)^6=0,
\end{equation}
with weighted degree~$6$. The intersection point of $\mathcal{C}_2$ and $\mathcal{C}_3$ is $[1:-1:0]_\omega$
and  $\mathcal{C}_2$ has a nodal point on $\mathcal{F}_0$ at
$[1:0:1]_\omega$; in particular, $\mathcal{C}_2\cup \mathcal{F}_0$ has an ordinary triple point there. The curve 
$\mathcal{C}_2$ has another double point  $[0:-1:1]_\omega$ (in $\mathcal{F}_\infty$).

Among the pencil of \emph{lines} through $[0:0:1]_\omega$, those with equations $x_\omega=0$, $y_\omega=0$ and $x_\omega+y_\omega=0$
intersect $\mathcal{C}_2\cup\mathcal{C}_3$ in two points.
A calculation shows that the other lines with this property are the tangent lines to $\mathcal{C}_2\cup\mathcal{C}_3$, with equations  $x_\omega+ 3y_\omega=0$ and $3 x_\omega+y_\omega=0$, and the tangency points
have quasi-homogeneous coordinates $[3:-1:-8]_\omega$ and  $[1:-3:-8]_\omega$.  These lines would be the vertical tangent lines to $\mathcal{C}_2$ in Figure~\ref{fig:sigma3}.

The local equation of $\mathcal{C}_2$ at the singular point  $[1:0:1]_\omega$ can be described in local coordinates $(u,v)\mapsto[1:u:v+1]_\omega$, and one finds the tangent cone is $0=21 u^2 - 6u v + v^2$.  These two lines are not real, so in the real picture one has an isolated point and can see $\mathcal{F}_0$ but not 
$\mathcal{C}_2$.

We can consider the
affine chart $\mathbb{C}^2\equiv\Sigma_3\setminus(E\cup\mathcal{F}_\infty)$, or equivalently, the affine chart
$(y_\omega,z_\omega)\mapsto[1:y_\omega:z_\omega]_\omega$ of of $\mathbb{P}^2_\omega$. Figure~\ref{fig:sigma3} shows a real picture of this affine chart 
$x_\omega=1$.

\begin{figure}[ht]
\centering
\begin{tikzpicture}[xscale=3]
\draw (-2,0)--(1,0);
\draw (0,-2.5)--(0,1);
\draw (-1,-1) ellipse [x radius=.5cm,y radius=1cm];
\draw[dotted] (-.5,-1) to[out=0,in=-90] (-.25,0)[out=90,in=180] to (0,.75);
\fill (0,.75) ellipse [x radius=.03333, y radius=.1];
\draw[dashdotted] (-.75,-2.5)--(-.75,1) node[above] {$\mathcal{F}_*$};
\draw[dashdotted] (-1.25,-2.5)--(-1.25,1) node[above] {$\mathcal{F}_-$};
\draw[dashdotted] (-.40,-2.5)--(-.40,1) node[above] {$\mathcal{F}_+^1$};
\draw[dashdotted] (-.15,-2.5)--(-.15,1) node[above] {$\mathcal{F}_+^2$};
\node[left] at (-1.5,-1) {$\mathcal{C}_2$};
\node[left] at (-2,0) {$\mathcal{C}_3$};
\node[right] at (0,-1) {$\mathcal{F}_0$};
\end{tikzpicture}
\caption{Affine chart $(y_\omega,z_\omega)$ of $\mathbb{P}^2_\omega$. The dotted line represents
real parts of the $y_\omega$-coordinates of the strict transform of the conic.}
\label{fig:sigma3}
\end{figure}
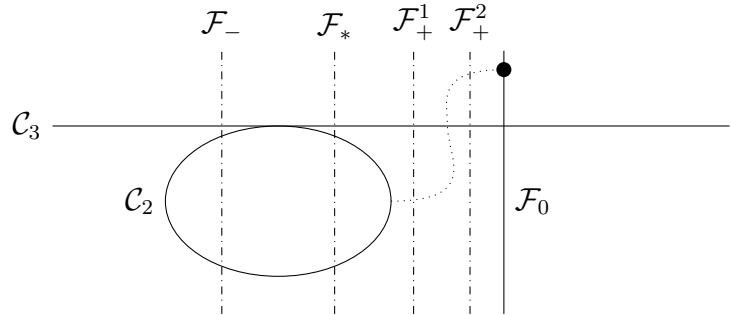

The base point for the fundamental group is in $\mathcal{F}_*=\{y_\omega=y_*\}$, with $z_\omega$-coordinate a real number $z_*\gg 1$. The geometric basis $c,q_1,q_2$ in this fibre,
plus a meridian of $\mathcal{F}_0$ lying on the horizontal line $z_\omega=z_*$, together generate the fundamental group of 
\[
\mathbb{P}^2\setminus(\mathcal{C}_2\cup\mathcal{C}_3\cup\mathcal{T}_\infty\cup\mathcal{T}_0)\cong
\Sigma_3\setminus(\mathcal{C}_2\cup\mathcal{C}_3\cup E\cup \mathcal{F}_\infty\cup \mathcal{F}_0)\cong
\mathbb{C}^2\setminus(\mathcal{C}_2\cup\mathcal{C}_3\cup \mathcal{F}_0).
\]
Let $y_-,y_+^1,y_+^2\in\mathbb{R}$ such that $\mathcal{F}_-=\{y_\omega=y_-\}$,
$\mathcal{F}_+^1=\{y_\omega=y_+^1\}$,
and $\mathcal{F}_+^2=\{y_\omega=y_+^2\}$. We consider geometric bases in these fibres which can be considered to have $(y_*,z_*)$
as base points. This is done if we take in the horizontal
line $z_\omega=z_*$ paths connecting $(y_*,z_*)$
with the base points in each vertical fiber, namely $(y_-,z_*)$, $(y_+^1,z_*)$, and $(y_+^2,z_*)$, see 
the upper part of Figure~\ref{fig:braids}.

\begin{figure}[ht]
\begin{tikzpicture}[scale=.5]
\tikzset{->-/.style={decoration={
  markings,
  mark=at position #1 with {\arrow[scale=1.5]{>}}},postaction={decorate}}}
\draw (-4,-2) rectangle (6,2);
\node[below left] at (6,2) {$y_\omega=y_*$};
\fill (5,0) circle [radius=.15] node[below=3pt] {$z_*$};
\fill (1,0) circle [radius=.15] node[below=5pt] {$c$};
\draw[->-=.76] (5,0) -- (1.5,0) arc [start angle=0,end angle=360, radius=.5];
\fill (-1,0) circle [radius=.15] node[below=5pt] {$q_1$};
\draw[->-=.83] (5,0) to[out=135, in=45] (-.5,0) arc [start angle=0,end angle=360, radius=.5];
\fill (-3,0) circle [radius=.15] node[below=5pt] {$q_2$};
\draw[->-=.865] (5,0) to[out=135, in=45] (-2.5,0) arc [start angle=0,end angle=360, radius=.5];
\begin{scope}[xshift=11cm]
\draw (-4,-2) rectangle (6,2);
\node[below left] at (6,2) {$y_\omega=y_-$};
\fill (5,0) circle [radius=.15] node[below=3pt] {$z_*$};
\fill (1,0) circle [radius=.15] node[below=5pt] {$c^{(q_1 c)^3}$};
\draw[->-=.76] (5,0) -- (1.5,0) arc [start angle=0,end angle=360, radius=.5];
\fill (-1,0) circle [radius=.15] node[below=5pt] {$q_1^{(q_1 c)^3}$};
\draw[->-=.83] (5,0) to[out=135, in=45] (-.5,0) arc [start angle=0,end angle=360, radius=.5];
\fill (-3,0) circle [radius=.15] node[below=5pt] {$q_2$};
\draw[->-=.865] (5,0) to[out=135, in=45] (-2.5,0) arc [start angle=0,end angle=360, radius=.5];
\end{scope}
\begin{scope}[yshift=-5cm]
\draw (-4,-1.75) rectangle (6,2.5);
\node[below left] at (6,2.5) {$y_\omega=y_+^1$};
\fill (5,0) circle [radius=.15] node[below=3pt] {$z_*$};
\fill (1,0) circle [radius=.15] node[below=5pt] {$c$};
\draw[->-=.76] (5,0) -- (1.5,0) arc [start angle=0,end angle=360, radius=.5];
\fill (-1,1) circle [radius=.15] node[left=15pt] {$q_1$};
\draw[->-=.83] (5,0) to (-.5,1) arc [start angle=0,end angle=360, radius=.5];
\fill (-1,-1) circle [radius=.15] node[left=15pt] {$q_2$};
\draw[->-=.865] (5,0) to[out=135, in=0] (-1,1.75) to[out=180,in=90] (-2.5,.5) to[out=-90,in=90] (-1,-.5) arc [start angle=90,end angle=450, radius=.5];
\end{scope}
\begin{scope}[yshift=-5cm,xshift=11cm]
\draw (-4,-1.75) rectangle (6,2.5);
\node[below left] at (6,2.5) {$y_\omega=y_+^2$};
\fill (5,0) circle [radius=.15] node[below=3pt] {$z_*$};
\fill (0,1) circle [radius=.15] node[below right=5pt] {$q_1$};
\draw[->-=.83] (5,0) to (.5,1) arc [start angle=0,end angle=360, radius=.5];
\fill (0,-1) circle [radius=.15] node[right=5pt] {$q_2^{q_1cq_1^{-1}}$};
\draw[->-=.865] (5,0) to[out=135, in=0] (0,1.6) to[out=180,in=90] (-1,.5) to[out=-90,in=90] (0,-.5) arc [start angle=90,end angle=450, radius=.5];
\fill (-2,0) circle [radius=.15] node[below=5pt] {$q_1cq_1^{-1}$};
\draw[->-=.76] (5,0) to[out=130,in=0] (0,1.75) to[out=180,in=90] (-2,.5) arc [start angle=90,end angle=450, radius=.5];
\end{scope}
\end{tikzpicture}
\caption{Geometric bases at the fibres; superindices stand for conjugation.}
\label{fig:gbf}
\end{figure}
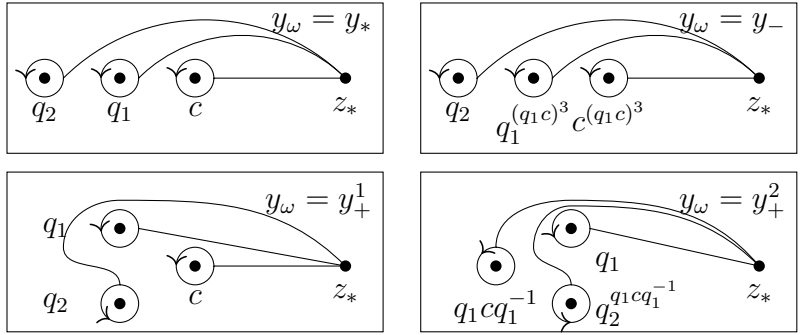

The elements of the geometric bases in each vertical line (Figure~\ref{fig:gbf}) can be expressed in terms of the generators
in $\mathcal{F}_*$. The expression of each of this element in terms of the generators
$c,q_1,q_2$ in $\mathcal{F}_*$ is shown in Figure~\ref{fig:gbf}. These equalities are obtained 
by the action of the connecting braids in the lower part of Figure~\ref{fig:braids} which defined isomorphisms
of the fundamental group of the punctured line $\mathcal{F}_*$ with the fundamental group 
of the punctured lines $\mathcal{F}_-$, $\mathcal{F}_-$, $\mathcal{F}_+^1$ and $\mathcal{F}_+^2$. 
In order to draw the connecting braids, when two points have
the same real part, we put the one with positive imaginary
part to the left of the one with negative imaginary part.

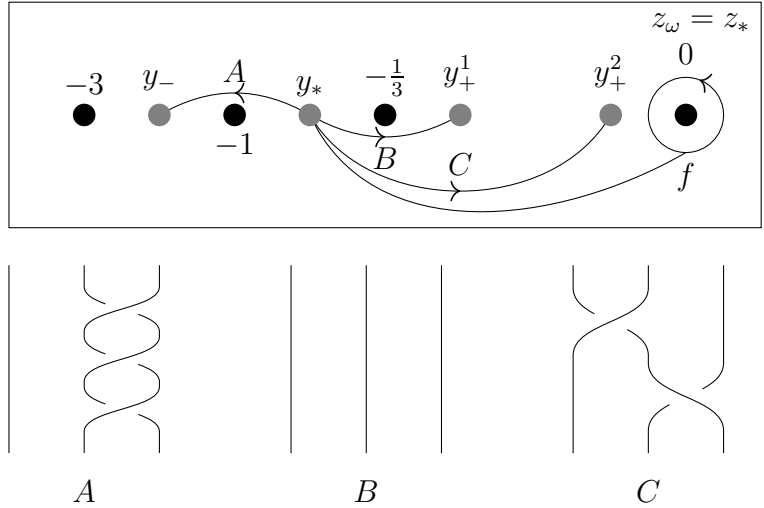
\begin{figure}[ht]
\centering
\begin{tikzpicture}[scale=1]
\tikzset{->-/.style={decoration={
  markings,
  mark=at position #1 with {\arrow[scale=1.5]{>}}},postaction={decorate}}}
\draw (-4,-1.5) rectangle (6,1.5);
\node[below left] at (6,1.5) {$z_\omega=z_*$};

\draw[->-=.5] (0,0) to[out=150,in=30] node [above, pos=.5] {$A$} (-2,0) ;

\draw[->-=.5] (0,0) to[out=-30,in=-150] node [below, pos=.5] {$B$}  (2,0);

\draw[->-=.5] (0,0) to[out=-60, in=-120] node [above=2pt, pos=.5] {$C$}  (4,0);

\draw[->-=.8] (0,0) to[out=-65, in=-150] (5,-.5) arc [start angle=-90,end angle=270, radius=.5] node[below] {$f$};
\fill[fill=gray] (0,0) circle [radius=.15] node[above=3pt] {$y_*$};
\fill[fill=gray] (-2,0) circle [radius=.15] node[above=5pt] {$y_-$};
\fill[fill=gray] (4,0) circle [radius=.15] node[above=5pt] {$y_+^2$};
\fill[fill=gray] (2,0) circle [radius=.15] node[above=5pt] {$y_+^1$};
\fill[] (-1,0) circle [radius=.15] node[below=3pt] {$-1$};
\fill[] (1,0) circle [radius=.15] node[above=3pt] {$-\frac{1}{3}$};
\fill[] (-3,0) circle [radius=.15] node[above=3pt] {$-3$};
\fill[] (5,0) circle [radius=.15] node[above=15pt] {$0$};
\begin{scope}[yshift=-2cm, xshift=3.5cm]
\pic[braid/.cd, 
gap=.15
]
{braid = {s_1^{-1} s_2}};
\node at (1,-3) {$C$};
\end{scope}
\begin{scope}[yshift=-2cm, xshift=-.25cm]
\pic[braid/.cd, number of strands=3,
height = -2cm
]
{braid = {1}};
\node at (1,-3) {$B$};
\end{scope}
\begin{scope}[yshift=-2cm, xshift=-4cm]
\pic[braid/.cd, number of strands=3,
height = -.666cm, gap =.15pt
]
{braid = {s_2^{-1} s_2^{-1} s_2^{-1}}};
\node at (1,-3) {$A$};
\end{scope}
\end{tikzpicture}
\caption{Paths in $y_\omega=y_*$ avoiding the non transversal vertical
lines and associated braids}
\label{fig:braids}
\end{figure}

The fundamental group is generated by $c,q_1,q_2$ in $\mathcal{F}_*$
(Figure~\ref{fig:gbf}) and $f$ (Figure~\ref{fig:braids}).
The first relation is obtained by turning around $y_\omega=-\frac{1}{3}$
(vertical tangency)
and it is $q_1=q_2$. For the sake of brevity, we set
$q:=q_1=q_2$. 

Turning around $y_\omega=-1$, as the singular point is simple of type $A_{11}$,
then 
the relation is $[q,(q\cdot c)^6]=1$.

The next relation comes from turning around $y_\omega=-3$. This is a 
again a vertical ordinary tangency, and we obtain the equality
of the second and third meridians if $\mathcal{F}_-$, i.e., 
$q_2=(q_1\cdot c)^{-3}\cdot q_1\cdot (q_1\cdot c)^3$ which can be expressed as 
$[q,(q\cdot c)^3]=1$. The previous relation becomes a consequence of this one.

Since $\mathcal{F}_0$ is part of the curve, the relations obtained by turning around $y_\omega=0$ involve
also the generator $f$ together 
with the meridians in $\mathcal{F}_{+}^2$. We obtain:
\[
f^{-1}\cdot(q_1\cdot c\cdot q_1^{-1})\cdot f=(q_1\cdot c\cdot q_1^{-1}),\text{ i.e. }
[f,q\cdot c\cdot q^{-1}]=1,
\]
and 
\[
[f,q_1\cdot c^{-1}\cdot q_1^{-1}\cdot q_2\cdot q_1\cdot c\cdot q_1^{-1},q_1]=1\Longleftrightarrow
[f,q\cdot c^{-1}\cdot q\cdot c\cdot q^{-1},q]=1\Longleftrightarrow
[f,q,c^{-1}\cdot q\cdot c]=1.
\]
Recall that this relation means that $f\cdot q\cdot c^{-1}\cdot q\cdot c$
commute with the three factors.

We are interested in the fundamental group of $\mathbb{P}^2\setminus(\mathcal{C}_2\cup\mathcal{C}_3\cup\mathcal{T}_\infty)$.
Let $\hat{\Sigma}_3$ be the space obtained by blowing up; it turns out that the exceptional component is
the strict transform of $\mathcal{T}_0$ and that $\mathbb{P}^2\setminus(\mathcal{C}_2\cup\mathcal{C}_3\cup\mathcal{T}_\infty)$
is isomorphic to the complement of $\mathcal{C}_2\cup\mathcal{C}_3\cup E\cup \mathcal{F}_\infty\cup \mathcal{F}_0$ in $\hat{\Sigma}_3$.
Considering the meridian of $\mathcal{T}_0$, by Proposition~\ref{prop:fujita} this means that
$f\cdot q\cdot c^{-1}\cdot q\cdot c=1$. 
Summarizing the group is generated by $q,c,f$ with relations
\[
[q,(q\cdot c)^3]=1,\quad [f,q\cdot c\cdot q^{-1}]=1,\quad f\cdot q\cdot c^{-1}\cdot q\cdot c=1.
\]
The third relation allows one to solve for~$f$; inserting the value of~$f$ into the second relation, one deduces that $[q,cqc]=1$; 
then, applying this new relation to the first relation written out, one sees that $qc=cq$.  Thus, the group is abelian, isomorphic to~$\mathbb{Z}\oplus\mathbb{Z}$.
\end{proof}

\subsection{Fundamental group of the complement of the Milnor fiber of a smoothing of \texorpdfstring{$\mathcal{C}^3_3(p)$}{C33(p)}}\label{sec:c33}
\mbox{}

In this subsection we consider the curve
$\mathcal{C}_2\cup\mathcal{C}_3$, to be used for the construction
of the $\qhd$-Milnor fiber for the family $\mathcal{C}^3_3(p)$~\cite[(8.6)]{SSW}.
In Figure~\ref{fig:C33-infty} we have a resolution of the singularities of~$\mathcal{C}_2\cup\mathcal{C}_3$ with extra blowing-ups at one of the branches of the node. The $\qhd$-Milnor fiber for the family $\mathcal{C}^3_3(p)$ is obtained by
forgetting the last exceptional component (gray vertex).

\begin{figure}[ht]
\begin{center}
\begin{tikzpicture}[scale=1.5]
\foreach \a in {-3,...,3}
{
\coordinate (A\a) at (\a,0);
}
\foreach \b in {2,...,6}
{
\coordinate (B\b) at (\b,1);
}
\coordinate (C) at (-1,-1);
\draw[dashed] (A-3)--(A-1);
\draw (A-3)--(A-2);
\draw (A-1)--(A3);
\foreach \a in {-3,0,1,...,3}
{
\fill (A\a) circle [radius=.1];
}
\foreach \b in {2,...,6}
{
\fill (B\b) circle [radius=.1];
}
\draw (A2)--(B2)--(B6);
\node[above=2pt] at (A1) {$-(p+3)$};
\node[below=2pt] at (A2) {$-1$};
\node[below=2pt] at (A3) {$\mathcal{C}_2$};
\node[below right=2pt] at (A1) {$\mathcal{C}_3$};
\node[above=2pt] at (A0) {$E_0$};
\node[above=2pt] at (A-3) {$E_{p+1}$};
\draw[dashed] (A-3)--(C)--(A1);
\fill[gray] (C) circle [radius=.1];
\node[above] at (C) {$-1$};

\end{tikzpicture}
\caption{Graph at infinity for $\mathcal{C}^3_3(p)$.}
\label{fig:C33-infty}
\end{center}
\end{figure}
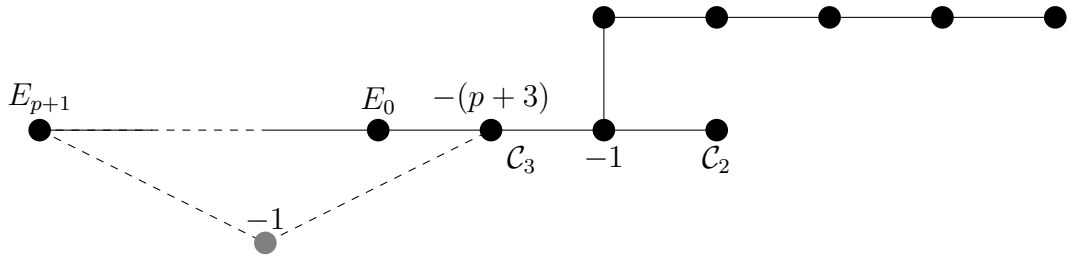

\begin{corollary}\label{cor:c33}
The Milnor fibre of the $\qhd$-smoothing of $\mathcal{C}^3_3(p)$ is abelian.
\end{corollary}

\begin{proof}
We work as in the proof of Corollary~\ref{cor:c23}.
The first step is to use Proposition~\ref{prop:fujita} to prove 
that the fundamental group of $\mathbb{P}^2\setminus(\mathcal{C}_2\cup\mathcal{C}_3)$
is abelian (in fact, isomorphic to $\mathbb{Z}$). Let $\pi:X\to\mathbb{P}^2$ be the (non-minimal) embedded resolution of the singularities
of $\mathcal{C}_2\cup\mathcal{C}_3$
depicted in Figure~\ref{fig:C33-infty}; we have that the fundamental group
of  $X\setminus\pi^{-1}(\mathcal{C}_2\cup\mathcal{C}_3)$
is abelian and we proceed as in the proof of Corollary~\ref{cor:c23}.
\end{proof}


\providecommand{\bysame}{\leavevmode\hbox to3em{\hrulefill}\thinspace}
\providecommand{\MR}{\relax\ifhmode\unskip\space\fi MR }
\providecommand{\MRhref}[2]{%
  \href{http://www.ams.org/mathscinet-getitem?mr=#1}{#2}
}
\providecommand{\href}[2]{#2}

\end{document}